\documentclass[aos,preprint,reqno]{imsart}

\RequirePackage{amsthm,amsmath,amsfonts,amssymb}
\RequirePackage[numbers]{natbib}
\RequirePackage[colorlinks,citecolor=blue,urlcolor=blue]{hyperref}
\RequirePackage{graphicx}
\RequirePackage{xcolor}
\RequirePackage{commath}

\startlocaldefs

\newtheorem{theorem}{Theorem}[section]
\newtheorem{proposition}[theorem]{Proposition}
\newtheorem{corollary}[theorem]{Corollary}
\newtheorem{lemma}[theorem]{Lemma}
\theoremstyle{remark}
\newtheorem{definition}[theorem]{Definition}

\ifx\example\undefined
\newtheorem{example}[theorem]{Example}
\newtheorem{remark}[theorem]{Remark}

\fi
\usepackage{slashbox}
\newcommand{\R}{{\mathbb R}}
\newcommand{\pfc}{{\mathcal F_c}}
\newcommand{\pkc}{{\mathcal K_c}}

\newcommand{\E}{{\mathbb E}}
\newcommand{\F}{{\mathbb F}}
\DeclareMathOperator{\argmin}{argmin}
\DeclareMathOperator{\midd}{mid}
\DeclareMathOperator{\spr}{spr}

\endlocaldefs

\begin{document}
	\renewcommand{\arraystretch}{1.2}
	
	\begin{frontmatter}
		\title{Projection depth and  $L^r$-type depths for fuzzy random variables}
		\runtitle{Projection and $L^{r}$-type Fuzzy Depth}
		
		\begin{aug}
			\author[A]{\fnms{LUIS} \snm{GONZ\'ALEZ-DE LA FUENTE}\ead[label=e1]{gdelafuentel@unican.es}},
			\author[A]{\fnms{ALICIA} \snm{NIETO-REYES}\ead[label=e2]{alicia.nieto@unican.es}}
			\and
			\author[B]{\fnms{PEDRO} \snm{TER\'AN}\ead[label=e3]{teranpedro@uniovi.es}}
			\address[A]{Departamento de Matem\'aticas, Estad\'istica y Computaci\'on,
				Universidad de Cantabria (Spain)
			}
			
			\address[B]{Universidad de Oviedo (Spain)
			}
		\end{aug}
		
		\begin{abstract}
			Statistical depth functions are a standard tool in nonparametric statistics to extend order-based univariate methods to the multivariate setting. Since there is no universally accepted total order for fuzzy data (even in the univariate case) and there is a lack of parametric models, a fuzzy extension of depth-based methods is very interesting. In this paper, we adapt projection depth and $L^{r}$-type depth to the fuzzy setting, studying their properties and illustrating their behaviour with a real data example.		
        \end{abstract}

		
	\end{frontmatter}

\section{Introduction}\label{intro}
It has repeatedly been observed (see, e.g., \cite{AngelesIJAR,GrzegorzewskiUltimo}) that statistical analysis of fuzzy data faces several difficulties:
\begin{itemize}
	\item[(a)]The algebraic structure of fuzzy sets, which is not a linear space and lacks a subtraction operation.
	\item[(b)]Fuzzy sets lack a natural total order (even in $\R$) and many competing approaches to ranking fuzzy numbers exist.
	\item[(c)]There is a substantial lack of parametric models and no practically useful analog of the normal distribution.
\end{itemize}
In this situation, nonparametric methods which are taylored to the specific structure of fuzzy set spaces and incorporate a well-founded way to order a fuzzy data sample would be very interesting. That is exactly what {\em statistical depth for fuzzy data} \cite{primerarticulo} tries to achieve.

By definition, the medians are the points with respect to which at least half of the sample is smaller or equal, and at least half of the sample is greater or equal. A seemingly innocuous rewording replaces ordering by geometry: the medians are the points that split the real line into two half-lines each of which contains at least half of the sample. The 10th percentile is more outlying because the two half-lines it defines divide the sample very unevenly.

With this idea, Tukey \cite{tukey} realized that, in order to extend the notion of {\em position} of a point in a sample to the multivariate setting, it suffices to replace half-lines by half-spaces. To each $x\in\R^p$, Tukey associated a depth value $D(x)$ calculated as the greatest lower bound of the proportion of the sample points contained in any half-space whose boundary passes through $x$. Like in the real line, if $D(x)$ is very small there exists a hyperplane through $x$ splitting the sample very unevenly. That is, $x$ is quite outlying. And $D(x)$ will be largest if the sample is split (by the worst-case hyperplane through $x$) as evenly as it is possible. Thus data themselves define a way to rank points according to their centrality or outlyingness, without requiring a total ordering in $\R^p$.

Tukey's data-driven center-outward ordering is not unique. In time, more ways to assess statistical depth were discovered and eventually Zuo and Serfling \cite{ZuoSerfling} proposed a list of desirable properties for a statistical depth function.
Depth functions in the literature often fail to satisfy all those properties perfectly. The dominant view is that this does not automatically disqualify a candidate depth function but it surely points out a weakness that should be taken into account in a practical context (see Remark \ref{qwert} in this regard). Therefore, understanding the theoretical properties of each depth function is an important step to make an informed choice between them.

In \cite{primerarticulo}, we proposed a defining list of desirable properties for statistical depth in the fuzzy case. Unlike with multivariate data, many different distances between fuzzy sets are available. Thus we suggested a definition of depth which only depends on the algebraic operations between fuzzy sets (semilinear depth functions) as well as a metric dependent definition (geometric depth functions) and studied the relationships between them. While there are approaches to depth in abstract metric spaces \cite{NietoBattey,Cho,Sara}, our definition (see Properties P1--P4b below) was conceived with the specificities of fuzzy data in mind, and in particular it would make sense for (crisp) set-valued data as well. In this connection, using statistical depth for either fuzzy or set-valued data was also independently proposed by Cascos {\it et al.} \cite{Cascos} and Sinova \cite{Sinova}.

This paper is part of an ongoing program to develop depth-based methods for fuzzy data. In \cite{primerarticulo}, besides proposing an abstract list of desirable properties we studied a generalization of Tukey depth to the fuzzy setting and showed that it fulfils all those properties. Next it becomes necessary to establish whether some popular, relevant statistical depth functions also admit adaptations and whether their properties are preserved in this more general setting. 
Once a library of depth functions becomes available, comparing their performance for specific purposes will be possible.
In \cite{Simplicial} we studied several ways to adapt Liu's simplicial depth \cite{LiuSimplicial} and their properties. In this paper, projection depth and $L^r$-depth, initially defined in $\mathbb{R}^p,$ are similarly studied in the fuzzy setting. 

Projection depth  \cite[Example 2.4]{ZuoSerfling} of a point $x\in\mathbb{R}^p$ with respect to  the distribution of a random vector $X$  considers the projections of $x$ in every direction and compares them with the univariate median of the corresponding projection of the distribution. In that sense, it measures the worst case of outlyingness of $x$ with respect  to the median of the distribution in any direction. 

It is formally defined as $$PD(x; X) := \left(1+O\left(x; X\right)\right)^{-1},$$ with
\begin{equation}\label{Om}
O(x; X) := \sup_{u\in\mathbb{S}^{p-1}}\cfrac{\left|\langle x,u\rangle - \text{med}\left(\langle x, X\rangle\right)\right|}{\mbox{MAD}\left(\langle u,X \rangle\right)}.
\end{equation}
In \eqref{Om},  $\langle \cdot ,\cdot\rangle$ denotes the usual inner product in $\mathbb{R}^{p},$ and $\mathbb{S}^{p-1} := \{x\in\mathbb{R}^{p}: \|x\|\leq 1 \}$ the unit sphere, with $\|.\|$ the Euclidean norm on $\mathbb{R}^{p}$. Moreover, $\text{med}(Y)$ and $\mbox{MAD}(Y)$ denote the median and the median absolute deviation of a random variable $Y$. Notice the set of all medians  will be denoted by $\text{Med}(Y)$ and the usual convention of defining $\text{med}(Y)$ to be the midpoint of $\text{Med}(Y)$ applies.

The function $O,$ which measures the outlyingness of a point with respect to the median, is widely considered in the literature. For instance, in the univariate case it appears in \cite{mostellertukey} and the multivariate version was introduced in \cite{donohogasko}.

The {\em $L^{r}$-depth} \cite[Example 2.3]{ZuoSerfling} of  $x\in\mathbb{R}^p$ with respect to the distribution of a random vector $X$ is 
\begin{equation}\label{Lm}
L^{r}D(x; X) := \left(1 + \text{E}[\|x-X\|_{r}]\right)^{-1},
\end{equation}
where
$\text{E}[\cdot]$ denotes the expected value
and  $\|\cdot\|_{r}$ is the $r$-norm in $\mathbb{R}^{p}$ (the same notation will be used for the  $L^{r}$-norm in function spaces). The structure is similar to that of projection depth, but now the function $\text{E}[\|\cdot-X\|_{r}]$ measures the distance from a point to the distribution.

	This paper is organized as follows. Section \ref{prelim} contains the notation and basic results on fuzzy sets, fuzzy random variables and statistical depth. The generalization of  projection depth and the study of the desirable properties from \cite{primerarticulo}  is presented in Section \ref{Sp}. The different notions of $L^{r}$-type depths for fuzzy sets and the study of their properties are proposed in Section \ref{Sl}. An example of real fuzzy data is analyzed  in Section \ref{simulations}.  All proofs are deferred to Section \ref{proofs}. Some final remarks close the paper in Section \ref{discussion}.

\section{Notation and preliminaries}
\label{prelim}

\subsection{Fuzzy sets}
A function $A:\mathbb{R}^p\rightarrow [0,1]$ is called a {\em fuzzy set} on $\mathbb{R}^p$.
Let $\alpha\in (0,1]$, the {\em $\alpha$-level} of a fuzzy set $A$ is defined to be $A_{\alpha}:= \{x\in\mathbb{R}^{p}: A(x)\geq\alpha \}$ and $A_0 = \text{clo}(\{x\in\mathbb{R}^p: A(x) > 0\})$, where $\text{clo}(\cdot)$ denotes the closure of a set. By $\mathcal{F}_{c}(\mathbb{R}^{p})$ we denote the set of all fuzzy sets $A$ on $\mathbb{R}^{p}$ whose $\alpha$-level is a non-empty compact and convex set for each $\alpha\in [0,1]$. 
For simplicity, we will just refer to the elements of $\mathcal{F}_{c}(\mathbb{R}^{p})$ as fuzzy sets, although a general fuzzy set may not be in $\pfc(\R^p)$.

Let $\mathcal{K}_{c}(\mathbb{R}^{p})$ denote the class of all non-empty  compact and convex subsets of $\mathbb{R}^{p}$.
Any set $K\in\mathcal{K}_{c}(\mathbb{R}^{p})$ can be identified with a fuzzy set via its indicator function $\text{I}_{K} : \mathbb{R}^{p}\rightarrow [0,1]$, where $\text{I}_{K}(x) = 1$ if $x\in K$ and $\text{I}_{K}(x) = 0$ otherwise. For any $K\in\pkc(\R^p)$, define $\|K\|=\max_{x\in K}\|x\|$.

The {\em support function} of a fuzzy set $A$ is the mapping $s_{A}: \mathbb{S}^{p-1}\times [0,1]\rightarrow\mathbb{R}$ defined by
$
s_{A}(u,\alpha) := \sup_{v\in A_{\alpha}}\langle u,v\rangle,
$ 
for every $u\in\mathbb{S}^{p-1}$ and $\alpha\in [0,1].$

	In $\mathcal{F}_{c}(\mathbb{R})$ it is common to use \textit{trapezoidal fuzzy numbers} (e.g., \cite[Section 10.7]{trape}). For any real numbers $a\leq b\leq c \leq d,$  the fuzzy set given by
	
	\begin{equation}\label{trapezoidal}
		\mbox{Tra}(a,b,c,d)(x) := \left\{ \begin{array}{lcc}
			\cfrac{x - a}{b-a}  & \text{ if } x\in [a,b),\\
			\\1& \text{ if } x\in [b,c] ,\\
			\\ \cfrac{x - c}{d-c}  &  \text{ if } x\in (c,d],\\ 
			\\ 0 & \text{otherwise}
		\end{array}
		\right.
	\end{equation}
is called a trapezoidal fuzzy number.

\subsection{Arithmetics and Zadeh's extension principle}\label{Aarith}
Let $A,B\in\mathcal{F}_{c}(\mathbb{R}^{p})$ and $\gamma\in\mathbb{R}$. According to \cite{zadehfuzzysets}, the operations \textit{sum} and \textit{product by a scalar} are defined by
\begin{equation}\nonumber
	(A + B)(t) := \sup_{x,y\in\mathbb{R}^{p}: \mbox{ } x + y = t} \min\{A(x), B(y) \}, \text{ with } t\in\mathbb{R}^{p},
\end{equation}
\begin{equation}\nonumber
	(\gamma\cdot A)(t) := \sup_{x\in\mathbb{R}^{p} : \mbox{ } t = \gamma\cdot x} A(x) = \left\{
	\begin{array}{ll}
		A\left(\frac{t}{\gamma}\right),     & \mathrm{if\ } \gamma\neq 0 \\
		\\
		I_{\{0\}}(t) & \mathrm{if\ } \gamma = 0
	\end{array},\text{ with } t\in\mathbb{R}^{p}.
	\right.
\end{equation}
Given $A,B\in\mathcal{F}_{c}(\mathbb{R}^{p})$, $\gamma\in [0,\infty)$, $u\in\mathbb{S}^{p-1}$ and $\alpha\in [0,1],$  a useful  relationship between the support function and these operations 
is the formula
\begin{equation}\label{soportesuma}
	s_{A+\gamma\cdot B}(u,\alpha) = s_{A}(u,\alpha) + \gamma\cdot s_{B}(u,\alpha).
\end{equation}

The $(\midd/\spr)$-decomposition is a commonly used tool to deal with support functions of fuzzy sets. Given $A\in\mathcal{F}_{c}(\mathbb{R}^{p})$ and $s_{A}$ the support function of $A$, it can be expressed as

\begin{equation}\label{sumasoporte}
	s_{A}(u,\alpha) = \midd (s_{A})(u,\alpha) + \spr(s_{A})(u,\alpha),
\end{equation}

where, for all $u\in\mathbb{S}^{p-1}$ and $\alpha\in[0,1],$

\begin{equation}\label{midspr}
	\begin{aligned}
		\midd (s_{A})(u,\alpha) := (s_{A}(u,\alpha) - s_{A}(-u,\alpha))/2,\\ 
		\spr (s_{A})(u,\alpha) := (s_{A}(u,\alpha) + s_{A}(-u,\alpha))/2.
	\end{aligned}
\end{equation}

 A function $f:\pfc(\R^p)\to \R$ is {\em convex} if $$f(\lambda\cdot A+(1-\lambda)\cdot B)\le \lambda\cdot f(A)+(1-\lambda)\cdot f(B)$$ for all $\lambda\in[0,1]$ and $A,B\in\pfc(\R^p)$.

Zadeh's extension principle  \cite{zadehextension} allows to apply a crisp function  $f: \mathbb{R}^{p}\rightarrow\mathbb{R}^{p}$ to a fuzzy set $A\in\mathcal{F}_{c}(\mathbb{R}^{p}),$ obtaining a new fuzzy set $f(A)\in\mathcal{F}_{c}(\mathbb{R}^{p})$  with $$f(A)(t) := \sup\{A(y) : y\in\mathbb{R}^{p}, f(y) = t \}$$
for all $t\in\mathbb{R}^{p}$.

	Let $M\in\mathcal M_{p\times p}(\mathbb R)$ be a regular matrix, $A\in\mathcal F_c(\mathbb R^p)$ a fuzzy set and let $f:\mathbb R^p\rightarrow\mathbb R^p$ be the function given by $f(x) = M\cdot x$. 
	The application of Zadeh's extension principle results in the fuzzy set $f(A) = M \cdot A$ defined as
	$$ 
	(M\cdot A)(t) = \sup\{A(y) : y\in\mathbb R^p, M\cdot y = t\}
	$$

From \cite[Proposition 7.2]{primerarticulo},
\begin{equation}\label{pa}
	s_{M\cdot A}(u,\alpha) = \left\|M^{T}\cdot u\right\|\cdot s_{A}\left(\cfrac{1}{\left\|M^{T}\cdot u\right\|}\cdot M^{T}\cdot u,\alpha\right)
\end{equation}
for any $A\in\mathcal{F}_{c}(\mathbb{R}^{p}),$  $M\in\mathcal{M}_{p\times p}(\mathbb{R})$ a regular matrix, $u\in\mathbb{S}^{p-1}$ and $\alpha\in [0,1].$

\subsection{Metrics between fuzzy sets}
\label{MFS}
Given fuzzy sets $A,B\in\mathcal{F}_{c}(\mathbb{R}^{p})$, define
\begin{equation}\label{ddr}
	d_{r}(A,B) := \left\{ \begin{array}{lcc}
		\left(\int_{[0,1]} \left( d_H (A_{\alpha},B_{\alpha}) \right)^{r} \dif\nu (\alpha)\right)^{1/r}, &  r\in [1,\infty)\\
		\\ \sup_{\alpha\in [0,1]} d_H  (A_{\alpha},B_{\alpha}), &   r = \infty,
	\end{array}
	\right.
\end{equation}

where 
$$d_H (S,T) := \max\left\{\sup_{s\in S}\inf_{t\in T} \parallel s-t\parallel, \sup_{t\in T}\inf_{s\in S}\parallel s-t\parallel\right\}$$ 
is the {\em Hausdorff metric}  between elements of $\mathcal{K}_{c}(\mathbb{R}^{p})$ and $\nu$ denotes the Lebesgue measure over $[0,1].$ 
  The metric space $(\mathcal{F}_{c}(\mathbb{R}^{p}), d_{r})$ is separable and non-complete for any $r\in(1,\infty)$, while the metric space $(\mathcal{F}_{c}(\mathbb{R}^{p}), d_{\infty})$ is non-separable and complete  \cite{diamondkloden}.

$L^{r}$-type metrics can be considered using the support function \cite{diamondkloden}. Given $A,B\in\mathcal{F}_{c}(\mathbb{R}^{p})$  and $r\geq 1,$
\begin{equation}\label{rhor}
	\rho_{r}(A,B) := \left(\int_{[0,1]}\int_{\mathbb{S}^{p-1}}|s_{A}(u,\alpha)-s_{B}(u,\alpha)|^{r}\dif\mathcal{V}_{p}(u)\, \dif\nu(\alpha)\right)^{1/r},
\end{equation}
where $\mathcal{V}_{p}$ denotes the normalized Haar measure in $\mathbb{S}^{p-1}$.

\subsection{Fuzzy random variables}\label{Frv}

Let $(\Omega,\mathcal{A})$ be a measurable space.
A function $\Gamma:\Omega\rightarrow\mathcal{K}_{c}(\mathbb{R}^{p})$ is a \emph{random compact set} \cite{Mol} if $\{\omega\in\Omega : \Gamma(\omega)\cap K\neq\emptyset \}\in\mathcal{A}$ for all $K\in\mathcal{K}_{c}(\mathbb{R}^{p})$, or equivalently if $\Gamma$ is Borel measurable with respect to the Hausdorff metric.
According to \cite{PuriRalescu}, a function $\mathcal{X}:\Omega\rightarrow\mathcal{F}_{c}(\mathbb{R}^{p})$ is called a \emph{fuzzy random variable} if the $\alpha$-level $\mathcal{X}_{\alpha}(\omega)$ is a random compact set for all $\alpha\in [0,1]$ where $\mathcal{X}_{\alpha}:\Omega\rightarrow\mathcal{K}_{c}(\mathbb{R}^{p})$ is defined as
$\mathcal{X}_{\alpha}(\omega) := \{x\in\mathbb{R}^{p}: \mathcal{X}(\omega)(x)\geq\alpha  \}$ for any $\omega\in\Omega$.

Let us denote by $L^{0}[\mathcal{F}_{c}(\mathbb{R}^{p})]$ the class of all fuzzy random variables on  $(\Omega,\mathcal{A})$. For any $r\in[1,\infty),$ we denote by $L^r[\pfc(\R^p)]$ the subset of  fuzzy random variables in $L^{0}[\mathcal{F}_{c}(\mathbb{R}^{p})]$  such that $E[\|\mathcal X_0\|^r]<\infty.$ Fuzzy random variables in $L^1[\pfc(\R^p)]$ are called {\em integrably bounded}.

The support function of a fuzzy random variable $\mathcal{X}$ is the function $s_{\mathcal{X}} : \mathbb{S}^{p-1}\times [0,1]\times\Omega\rightarrow\mathbb{R}$ with
$
s_{\mathcal{X}}(u,\alpha,\omega) := s_{\mathcal{X}(\omega)}(u,\alpha)
$
for all $u\in\mathbb{S}^{p-1}, \alpha\in [0,1]$ and $\omega\in\Omega$.
Throughout the paper, the probability space associated with a fuzzy random variable is denoted by $(\Omega,\mathcal{A},\mathbb{P})$.

\subsection{Symmetry and depth: semilinear and geometric notions}\label{notiondepth}

 In \cite{primerarticulo}, we proposed two notions of symmetry in the fuzzy setting, the $F$-symmetry notion, based in the support function, and the $(\midd,\spr)$-notion, based on the $(\midd,\spr)$-decomposition. 
Given a fuzzy random variable $\mathcal{X}:\Omega\rightarrow\mathcal{F}_{c}(\mathbb{R}^{p})$ and a fuzzy set $A\in\mathcal{F}_{c}(\mathbb{R}^{p})$, 

\begin{itemize}
	\item
	$\mathcal{X}$ is \emph{$F$-symmetric} with respect to $A$ if
	\begin{equation}\nonumber
		s_{A}(u,\alpha) - s_{\mathcal{X}}(u,\alpha) =^{d} s_{\mathcal{X}}(u,\alpha) - s_{A}(u,\alpha),
	\end{equation}
for all $(u,\alpha)\in\mathbb{S}^{p-1}\times[0,1],$ where $=^{d}$ represents being equal in distribution.
\item $\mathcal{X}$ is said to be $(\midd,\spr)$\emph{-symmetric} with respect to $A$ if
		\begin{equation}
			\begin{aligned} \nonumber
				\midd (s_{A}(u,\alpha)) - \midd (s_{\mathcal{X}}(u,\alpha)) &=^{d} \midd (s_{\mathcal{X}}(u,\alpha)) - \midd (s_{A}(u,\alpha)) \mbox{ and }\\
				\nonumber
				\spr (s_{A}(u,\alpha)) - \spr (s_{\mathcal{X}}(u,\alpha)) &=^{d} \spr (s_{\mathcal{X}}(u,\alpha)) - \spr (s_{A}(u,\alpha)).
			\end{aligned}
	\end{equation}
for all $(u,\alpha)\in\mathbb{S}^{p-1}\times[0,1]$.
\end{itemize}

There it is also proved that, for all $u\in\mathbb{S}^{p-1}$ and $\alpha\in [0,1],$
\begin{eqnarray}\label{Amedian}
	s_{A}(u,\alpha) \in \text{Med}(s_{\mathcal{X}}(u,\alpha))  \mbox{ if }\mathcal{X}  \mbox{ is } F\mbox{-symmetric with respect to }A
\end{eqnarray}
and		
	\begin{eqnarray}\label{lemamedian}			
		\midd (s_{A})(u,\alpha)\in \text{Med}(\midd (s_{\mathcal{X}})(u,\alpha))  \mbox{ and } \spr (s_{A})(u,\alpha)\in \text{Med}(\spr (s_{\mathcal{X}})(u,\alpha))
	\end{eqnarray}
	if $\mathcal{X}$ is $(\midd,\spr)$-symmetric with respect to $A$.

 In \cite{primerarticulo}, we introduced the following two abstract definitions of a statistical depth function for fuzzy data. Let us consider $\mathcal H\subseteq L^0[\pfc(\R^p)],$  $\mathcal{J}\subseteq\mathcal{F}_{c}(\mathbb{R}^{p})$ and a mapping  $D(\cdot;\cdot):\mathcal{J}\times
{\mathcal H}
\rightarrow[0,\infty).$ 
Let $A\in\mathcal{J}$ be such that  $D(A;\mathcal{X}) = \sup \{D(B;\mathcal{X}) : B\in\mathcal{J}\}$ and let $d:\mathcal{F}_{c}(\mathbb{R}^{p})\times\mathcal{F}_{c}(\mathbb{R}^{p})\rightarrow[0,\infty)$ be a metric.
Consider the following properties, which are required to hold for any such $A$.

\begin{enumerate}
	\item[{\bf P1}.] $D(M\cdot U + V; M\cdot\mathcal{X} + V) = D(U;\mathcal{X})$  for any regular matrix $M\in\mathcal{M}_{p\times p}(\mathbb{R}),$ any $U,V\in\mathcal{J}$ and any $\mathcal{X}\in{\mathcal{H}}.$
	\item[{\bf P2}.] For any symmetric fuzzy random variable  $\mathcal{X}\in{\mathcal H}$ (for some notion of symmetry),
	$D(U;\mathcal{X}) = \sup_{B\in\mathcal{F}_{c}(\mathbb{R}^{p})} D(B;\mathcal{X}),$ where $U\in\mathcal{J}$ is a center of symmetry of $\mathcal{X}.$
	\item[{\bf P3a}.]
	$
	D(A;\mathcal{X})\geq D((1-\lambda)\cdot A + \lambda\cdot U;\mathcal{X})\geq D(U;\mathcal{X})
	$
	for all $\lambda\in[0,1]$ and all $U\in\mathcal{F}_{c}(\mathbb{R}^{p})$.
	\item[{\bf P3b}.] $
	D(A;\mathcal{X})\geq D(U;\mathcal{X})\geq D(V;\mathcal{X})$		for all $B,C\in\mathcal{J}$ satisfying $d(A,V) = d(A,U) + d(U,V)$.
	\item[{\bf P4a}.] $
	\lim_{\lambda\rightarrow\infty} D(A + \lambda\cdot U;\mathcal{X}) = 0
	$ 	for all $U\in\mathcal{J}\setminus\{\text{I}_{\{0\}}\}$.
	\item[{\bf P4b}.] $
	\lim_{n\rightarrow\infty} D(A_{n};\mathcal{X}) = 0$ for every sequence  $\{ A_{n}\}_{n},$ with $A_n\in\mathcal{J}$ for all $n\in\mathbb{N},$
	 such that $d(A_{n},A) \to \infty$.
\end{enumerate}
These properties adapt to the specificities of fuzzy data the defining properties of a statistical depth function in multivariate analysis \cite{ZuoSerfling}.
As defined in \cite{primerarticulo},
 $D$  is  a \emph{semilinear depth function} if it satisfies P1, P2, P3a and P4a.
It  is  a \emph{geometric depth function} with respect to a metric $d$  if it satisfies P1, P2, P3b and P4b for that metric.

\subsection{Banach spaces}

A {\em Banach space} is a real normed space $(\E,\|\cdot\|)$ whose induced metric is complete.

\begin{definition}[\label{strictbanach}]
	Let $(\mathbb E,\|\cdot\|)$ be a Banach space. It is said to be {\em strictly convex} if $x = y$ whenever $\|(1/2)\cdot (x+y)\| = \|x\| = \|y\|$ for every $x,y\in\mathbb E$.
\end{definition}

The Cartesian product $\E\times\F$ of two Banach spaces $(\E,\|\cdot\|_\E)$ and $(\F,\|\cdot\|_\F)$ can be endowed with an $r$-norm
$$\|(x,y)\|_r=(\|x\|_\E^r+\|y\|_\F^r)^{1/r}.$$
The resulting Banach space is denoted by $\E\oplus_r\F$.

\section{Projection depth and its properties}
\label{Sp}\label{[[3]]}

\subsection{Definition}
In this section, we introduce a statistical depth function inspired by  multivariate projection depth. We extend the notion of projection depth by replacing in \eqref{Om} the product functionals $\langle u,\cdot\rangle$ by the support functionals $s_\cdot(u,\alpha)$. A rationale for this adaptation is given in \cite[Section 6]{primerarticulo}.

\begin{definition}
	The \emph{projection depth} based on $\mathcal{J}\subseteq\mathcal{F}_{c}(\mathbb{R}^{p})$  and $ \mathcal{H}\subseteq L^{0}[\mathcal{F}_{c}(\mathbb{R}^{p})]$ of a fuzzy set $A\in\mathcal{J}$ with respect to a fuzzy random variable  $\mathcal{X}\in\mathcal{H}$ is $$D_{FP}(A;\mathcal{X}) := \left(1 + O\left(A;\mathcal{X}\right)\right)^{-1},$$ where
	\begin{equation}\label{O}
		O(A;\mathcal{X}) := \sup_{u\in\mathbb{S}^{p-1}, \alpha\in [0,1]}\cfrac{|s_{A}(u,\alpha) - \text{Med}(s_A(u,\alpha))|}{\text{MAD}(s_{\mathcal{X}}(u,\alpha))}.
	\end{equation}
The usual convention of taking the mid-point of the interval of medians when the median is not unique is adopted, both in the numerator and the denominator.
\end{definition}
We consider the particular case of the function $D_{FP}$ based on $$\mathcal{J}=\left\{\text{I}_{\{x\}}\in\mathcal{F}_{c}(\mathbb{R}^{p}): x\in\mathbb{R}^{p}\right\},$$
showing $D_{FP}$ generalizes multivariate projection depth.
\begin{proposition}\label{promult}
	Let $\mathcal J = \{\text{I}_{\{x\}} : x\in\mathbb R^p\}$. For any random vector $X$ on $\mathbb{R}^{p}$  and any $x\in\mathcal J,$  
$$D_{FP}\left(\text{I}_{\{x\}};\text{I}_{\{X\}}\right) = PD(x;X).$$
\end{proposition}
The proof follows directly from the fact that $s_{A}(u,\alpha) = \langle u, x\rangle$ for any   $A = \text{I}_{\{x\}}$, $u\in\mathbb{S}^{p-1}$ and $\alpha\in [0,1].$

 \subsection{Properties}

We will now show that projection depth, like Tukey depth \cite{primerarticulo}, is both a semilinear depth function and a geometric depth function.

\begin{theorem}\label{theoremprojectionsemilinear}
	$D_{FP}$ satisfies  properties P1, P2 with $F$-symmetry,
 P3a and P4a. Moreover, it satisfies
 P3b for $\rho_r$ if $r\in(1,\infty)$ and
P4b for $\rho_{r}$ if $r\in [1,\infty)$ and $d_{r}$ if $r\in [1,\infty].$
\end{theorem}

\begin{corollary}\label{coropro}
	When using the $F$-symmetry notion, $D_{FP}$ is a semilinear depth function and a geometric depth function for the $\rho_{r}$ distance for any $r\in (1,\infty)$.
\end{corollary}

The next result shows that $D_{FP}$ is not a geometric depth function for the $d_{r}$ metrics. Using \cite[Example 5.6]{primerarticulo}, it is proved by counterexample that $D_{FP}$ violates property P3b for some metrics.

\begin{proposition}
	$D_{FP}$ is not a geometric depth function for the $d_{r}$-distance for any $r\in [1,\infty]$.
\end{proposition}

\section{$L^{r}$-type depths and their properties}\label{Sl}\label{[[4]]}

\subsection{Definitions}\label{Ld}

We present several approaches to statistical depth for fuzzy data inspired by multivariate $L^{r}$-depth.
 As is apparent from \eqref{Lm}, a distance between fuzzy data is required. A natural $L^r$-type distance is  the $\rho_{r}$ metric defined above.

\begin{definition}\label{definicionLDr}
	For any $r\in[1,\infty),$ the  \emph{$r$-natural depth} based on $\mathcal{J}\subseteq\mathcal{F}_{c}(\mathbb{R}^{p})$ and $ \mathcal{H}\subseteq L^{1}[\mathcal{F}_{c}(\mathbb{R}^{p})]$ of a fuzzy set $A\in\mathcal{J}$ with respect to a fuzzy random variable  $\mathcal{X}\in\mathcal{H}$ is
	\begin{equation}\nonumber
		D_{r}(A;\mathcal{X}) := \left(1 + \text{E}[\rho_{r}(A,\mathcal{X})]\right)^{-1}.
	\end{equation}
\end{definition}
The reason to consider $\mathcal{H}\subseteq L^{1}[\mathcal{F}_{c}(\mathbb{R}^{p})]$ is to avoid having an infinite expectation in the definition. While it is possible to define $D_{r}$ as being identically zero in that case  (see \cite[Example 5.9]{primerarticulo}), a null depth function is not desirable in practice, e.g., in classification problems.

\begin{definition}\label{definicionDr}
	For any $r\in[1,\infty),$ the   \emph{$r$-natural raised depth} based on $\mathcal{J}\subseteq\mathcal{F}_{c}(\mathbb{R}^{p})$ and $ \mathcal{H}\subseteq L^{r}[\mathcal{F}_{c}(\mathbb{R}^{p})]$ of a fuzzy set $A\in\mathcal{J}$ with respect to a random variable  $\mathcal{X}\in\mathcal{H}$ is
	\begin{equation}\nonumber
		RD_{r}(A;\mathcal{X}) := \left(1 + E[\rho_{r}(A,\mathcal{X})^{r}]\right)^{-1}.
	\end{equation}

\end{definition}

Another possibility is to define an $L^{r}$-type depth by using the $\midd$ and $\spr$ functions, through which the location and the shape of the fuzzy sets are described. With that aim, denoting by 
$\|\cdot\|_{r}$ the norm of the Banach space $L^{r}\left(\mathbb{S}^{p-1}\times [0,1],\mathcal{V}_{p}\otimes\nu\right),$ 
we define
\begin{equation}\label{dr}
	d_{r,\theta}(A,B) := \left[\left\|\midd (s_{A}) - \midd (s_{B})\right\|_{r}^{r} + \theta\cdot\left\|\spr (s_{A}) - \spr (s_{B})\right\|_{r}^{r} \right]^{1/r}
\end{equation}
for any $A,B\in\mathcal{F}_{c}(\mathbb{R}^{p})$, $r\in [1,\infty)$ and $\theta\in [0,\infty)$. This is a straightforward generalization of the distance $d_{2,\theta}$  in \cite{Trutsching}. For $\theta>0,$ $d_{r,\theta}$ is a metric, as it identifies isometrically each $A\in\pfc(\R^p)$ with the element $(\midd (s_A),\spr (s_A))$ of the Banach space
$$L^{r}\left(\mathbb{S}^{p-1}\times [0,1],\mathcal{V}_{p}\otimes\nu\right)\oplus_r L^{r}\left(\mathbb{S}^{p-1}\times [0,1],\theta^{1/r}\cdot(\mathcal{V}_{p}\otimes\nu)\right).$$
 In the case $\theta=0$ it depends only on $\midd$ and it is just a pseudometric. We will use this case for a counterexample (Proposition \ref{P4Dtheta}).

The definitions introduce a parameter $\theta$ in order to control the relative importance of the shape and location of the fuzzy sets. That resembles what happens in function spaces with the Sobolev distances.
As before, we give two proposals: one based on
$d_{r,\theta}$ and another on $d_{r,\theta}^r$.

\begin{definition}\label{definicionDrtheta}
	For any $r\in[1,\infty)$ and $\theta\in [0,\infty),$ the   \emph{$(r,\theta)$-location depth}  based on $\mathcal{J}\subseteq\mathcal{F}_{c}(\mathbb{R}^{p})$ and $ \mathcal{H}\subseteq L^{r}[\mathcal{F}_{c}(\mathbb{R}^{p})]$ of a fuzzy set $A\in\mathcal{J}$ with respect to a fuzzy random variable  $\mathcal{X}\in\mathcal{H}$ is
	\begin{equation}\nonumber
		D_{r}^{\theta}(A;\mathcal{X}) := \left(1 + \text{E}[d_{r,\theta}(A,\mathcal{X})]\right)^{-1}.
	\end{equation}
\end{definition}

\begin{definition}\label{definicionRDrtheta}
	For any $r\in[1,\infty)$ and $\theta\in [0,\infty),$ the \emph{$(r,\theta)$-location raised depth}  based on $\mathcal{J}\subseteq\mathcal{F}_{c}(\mathbb{R}^{p})$ and $ \mathcal{H}\subseteq L^{r}[\mathcal{F}_{c}(\mathbb{R}^{p})]$  of a fuzzy set $A\in\mathcal{J}$ with respect to a random variable $\mathcal{X}\in\mathcal{H}$ is
	\begin{equation}\nonumber
		RD_{r}^{\theta}(A;\mathcal{X}) := \left(1 + \text{E}[d_{r,\theta}(A,\mathcal{X})^{r}]\right)^{-1}.
	\end{equation}
\end{definition}

The particular case of $D_2^\theta$ in the real line was discussed in \cite[Section 6]{Sinova}.  Yet another similar definition, but involving only the spread and not the mid, is used in \cite[Example 5.7]{primerarticulo} to show that P3a does not imply P3b in general.

\begin{remark}
The general structure of the definitions above is
$$D(A;\mathcal X)=(1+\phi(E[d(A,\mathcal X)]))^{-1}$$
where $d$ is a metric in $\pfc(\R^p)$ and $\phi$ is an appropriate increasing (and convex, for some arguments in the sequel) function with $\phi(0)=0$. While this type of object makes sense in a general metric space, the next subsection will focus on whether it satisfies properties which are specific to the context of fuzzy sets.
\end{remark}

\begin{remark}
	Definitions \ref{definicionLDr} through \ref{definicionRDrtheta}
	adapt the multivariate notion of $L^r$-depth to the fuzzy setting but are not generalizations of it. The reason is that the $r$-norm distance between two points of $\R^p$ does not equal the $\rho_r$- or $d_{r,\theta}$-distance between their indicator functions. Take, for instance, $x = (2,3)$ and $y = (3,7).$ We have 
	$\|x-y\|_1=1+4=5$
	whereas
	$$\rho_{1}\left(I_{\{x\}},I_{\{y\}}\right)= \int_0^{2\pi}\left|\cos\theta+4\cdot\sin\theta\right|\dif\nu(\theta)= 4\sqrt{17}.$$
	Observe $d_{r,\theta}\left(\text{I}_{\{x\}},\text{I}_{\{y\}}\right)=\rho_{r}\left(\text{I}_{\{x\}},\text{I}_{\{y\}}\right)$ for all $\theta\in [0,\infty)$ since their spread is the null function.
\end{remark}

The following result states that functions of the form of $L^{r}$-type depths satisfy property P3a under certain convexity assumptions.

\begin{lemma}\label{teoremaZuo}\label{[[4.6]]}
	If $C(\cdot ,\mathcal{X})$ is a convex function then the function $(1 + \text{E}[C(\cdot ,\mathcal{X})])^{-1}$ satisfies P3a for every $\mathcal{X}\in L^{0}[\mathcal{F}_{c}(\mathbb{R}^{p})]$ such that $\text{E}[C(\text{I}_{\{0\}},\mathcal{X})] < \infty$.
\end{lemma}

This lemma and its proof are analogous to the multivariate result  \cite[Theorem 2.4]{ZuoSerfling}, since $C(\cdot,\mathcal{X})$ and P3a maintain the structure of their multivariate analogues.

\begin{proposition}\label{P3dtheta}
	Let $r\in[1,\infty)$, $\theta\in [0,\infty)$ and $\mathcal{X}\in L^{r}[\mathcal{F}_{c}(\mathbb{R}^{p})]$. The functions $\rho_{r}(\cdot;\mathcal{X})$, $\rho_{r}(\cdot;\mathcal{X})^{r}$, $d_{r,\theta}(\cdot;\mathcal{X})$ and $d_{r,\theta}(\cdot;\mathcal{X})^{r}$ are convex.
\end{proposition}

\subsection{Properties}
\subsubsection{Affine invariance}

The next example shows that neither  $D_{r},$ $RD_{r}$, $D_{r}^{\theta}$ nor $RD_{r}^{\theta}$ are affine invariant in the sense of property P1; the same  happens in the multivariate case \cite{ZuoSerfling}.

\begin{example}\label{M}
	Let $\{\{\omega_{1},\omega_{2}\},\mathcal{P}(\{\omega_{1},\omega_{2}\}),\mathbb{P}\}$ be a probability space with $\mathbb{P}(\{\omega_{1}\}) = \mathbb{P}(\{\omega_{2}\}) = 1/2$.
	
	\begin{itemize}
	\item[(i)]
		Let $\mathcal{X}(\omega_{1}) := \text{I}_{[1,2]}$ and $\mathcal{X}(\omega_{2}) := \text{I}_{[5,7]}.$
		Taking $A = \text{I}_{[3,4]}$, after some algebra we have, for any $r\in [1,\infty)$,

		\begin{equation}\nonumber
			\text{E}(\rho_{r}(A,\mathcal{X})) 
			= \cfrac{1}{2}\cdot\left[2 + \left(\cfrac{3^{r} + 2^{r}}{2}\right)^{1/r} \right]
		\end{equation}
		and
		\begin{equation}\nonumber
			\text{E}(\rho_{r}(A,\mathcal{X})^{r})
			= \cfrac{1}{2}\cdot\left[2^{r} + \cfrac{3^{r} + 2^{r}}{2} \right].
		\end{equation}
		Thus,
		\begin{equation}\nonumber
			D_{r}(A;\mathcal{X}) = \left(2 + \cfrac{1}{2}\cdot\left(\cfrac{3^{r} + 2^{r}}{2}\right)^{1/r}\right)^{-1} > 0
		\end{equation}
		and
		\begin{equation}\nonumber
			RD_{r}(A;\mathcal{X}) = \left(1 + 3\cdot\left(2^{r-2}  + \cfrac{3^{r-1}}{4}\right)\right)^{-1} > 0.
		\end{equation}
		Considering the matrix $M:=(5)\in\mathcal{M}_{1\times 1}(\mathbb{R}),$ 
		$$M\cdot\mathcal{X}(\omega_{1}) = \text{I}_{[5,10]}, \mbox{ } M\cdot\mathcal{X}(\omega_{2}) = \text{I}_{[25,35]} \mbox{ and } M\cdot A = \text{I}_{[15,20]}.$$
		Therefore, for every $r\in [1,\infty),$
		$$E[\rho_{r}(M\cdot A;M\cdot\mathcal{X}) = 5 E[\rho_{r}(A;\mathcal{X})]$$
		whence $D_{r}(M\cdot A;M\cdot\mathcal{X})\neq D_{r}(A;\mathcal{X})$ and $RD_{r}(M\cdot A;M\cdot\mathcal{X})\neq RD_{r}(A;\mathcal{X})$.
	\item[(ii)]Let
		$\mathcal{X}(\omega_{1}) := \text{I}_{[0,2]}$ and $\mathcal{X}(\omega_{2}) := \text{I}_{[2,3]}.$
		Taking $A = \text{I}_{[1,2]},$  we obtain  for any $r\in[1,\infty)$ and $\theta\in (0,\infty)$
		\begin{equation}\nonumber
			\text{E}[d_{r,\theta}(A,\mathcal{X})] = \cfrac{1}{2}\cdot\left(1 + \cfrac{(1+\theta)^{1/r}}{2}\right)
		\end{equation}
		and
		\begin{equation}\nonumber
			\text{E}[d_{r,\theta}(A,\mathcal{X})^{r}] = \cfrac{1}{2}\cdot\left(1 + \cfrac{1+\theta}{2^{r}}\right).
		\end{equation}
		Thus,
		\begin{equation}\nonumber
			D_{r}^{\theta}(A;\mathcal{X}) = \left(1 + \cfrac{1}{2}\cdot\left[1 + \cfrac{(1+\theta)^{1/r}}{2}\right]\right)^{-1} > 0
		\end{equation}
		and
		\begin{equation}\nonumber
			RD_{r}^{\theta}(A;\mathcal{X}) = \left(1 + \cfrac{1}{2}\cdot\left[1 + \cfrac{1+\theta}{2^{r}}\right] \right)^{-1} > 0.
		\end{equation}
		Now, for  $M = (2)\in\mathcal{M}_{1\times 1}(\mathbb{R}),$ 
$$M\cdot\mathcal{X}(\omega_{1}) = \text{I}_{[0,4]}, M\cdot\mathcal{X}(\omega_{2}) = \text{I}_{[4,6]} \mbox{ and }M\cdot A = \text{I}_{[2,4]}.$$ Therefore,
		\begin{equation}\nonumber
			\text{E}[d_{r,\theta}(M\cdot A,M\cdot\mathcal{X})] = 1 + \cfrac{(1+\theta)^{1/r}}{2}
		\end{equation}
		and
		\begin{equation}\nonumber
			\text{E}[d_{r,\theta}(M\cdot A,M\cdot\mathcal{X})^{r}] = 2^{r-1}\cdot\left(1 + \cfrac{1+\theta}{2^{r}} \right).
		\end{equation}
		For every $r\in [1,\infty)$ and $\theta\in (0,\infty)$,
		$$D_{r}^{\theta}(M\cdot A;M\cdot\mathcal{X}) \neq D_{r}^{\theta}(A;\mathcal{X}) \mbox{ and } RD_{r}^{\theta}(M\cdot A;M\cdot\mathcal{X}) \neq RD_{r}^{\theta}(A;\mathcal{X}).$$
	\end{itemize}
\end{example}

Let us consider the following property, weaker than P1.
\begin{enumerate}
	\item[{\bf P1$\ast$.}] $D(M\cdot A + B; M\cdot\mathcal{X} + B) = D(A;\mathcal{X})$ for any orthogonal matrix $M\in\mathcal{M}_{p\times p}(\mathbb{R})$ and $A,B\in\mathcal{F}_{c}(\mathbb{R}^{p}).$
\end{enumerate}
This property (called {\em rigid-body invariance}) was shown to hold in the multivariate case in \cite{ZuoSerfling}. 

The following result  states that $D_{r},$ $RD_{r},$ $D_{r}^{\theta}$ and $RD_{r}^{\theta}$ are invariant when the matrix $M\in\mathcal{M}_{p\times p}(\mathbb{R})$ is orthogonal.  That is due to the fact that $\|M^{T}\cdot u\| = 1$ for all $u\in\mathbb{S}^{p-1}$ if $M$ is orthogonal.
Note that the $M$'s in Example \ref{M}  are not  orthogonal matrices, because their determinant is not $\pm 1$.

\begin{proposition}\label{resultadoDr1}
Let $\mathcal{J}\subseteq\mathcal{F}_{c}(\mathbb{R}^{p})$, $\mathcal H_1 = L^{1}[\mathcal{F}_{c}(\mathbb{R}^{p})] $ and  $\mathcal{H}_r\subseteq L^{r}[\mathcal{F}_{c}(\mathbb{R}^{p})]$. 
Property  P1$\ast$ is satisfied by 	$D_{r}$ based on $\mathcal J$ and $\mathcal H_1$ and $RD_{r}$ based on $\mathcal{J}$ and $\mathcal{H}_r$, for any $r\in [1,\infty]$; and by $D_{r}^{\theta}$ based on $\mathcal J$ and $\mathcal H_1$ and $RD_{r}^{\theta}$ based on $\mathcal{J}$ and $\mathcal{H}$, for any $r\in [1,\infty]$ and $\theta\in [0,\infty).$
\end{proposition}

\begin{remark}\label{qwert}
	The failure of $P1$ and its multivariate analog for some non-orthogonal matrices illustrates why `lists of properties' are guides rather than axioms for depth functions.
	If the results of an analysis may be different depending on whether temperature values are expressed in the Celsius or Fahrenheit scale, one would like to ponder calmly whether it makes sense to use that method. Thus, failing affine invariance looks like an egregious violation for a depth function.
	
	From the discussion of Property P1*,  $L^r$-depths are rotation (and also translation) invariant, and only have problems with rescaling. Since both the function $x\mapsto (1+x)^{-1}$ and multiplication by a scalar are strictly monotonic, rescaling modifies the depth values but not their order. Therefore, as long as depth values are used as a ranking device (as opposed to important values in themselves) there will be no problem.
	
	For instance, consider a depth-trimmed mean obtained by eliminating from the sample the 10\% less $L^r$-deep points. Rescaling does not affect which sample points get trimmed and therefore the depth-trimmed mean will still be affinely invariant, even if the depth function itself is not. Similarly, a depth-based classification task will yield the same result regardless of rescaling.
	
	Moreover, in some situations data are routinely standardized before the analysis, which makes the rescaling issue irrelevant. For instance, in cell studies like cancer diagnosis, cell measurements  taken from tissue images need standardization since different images may not share the same scale.
\end{remark}

\begin{remark} 
	In \cite[Proposition 6.1]{Sinova}, Sinova shows what amounts to stating that $D_2^\theta$ (in the real line) satisfies Property P1*. In that case,  matrices are not involved since the only orthogonal transformations of $\R$ are the identity  function $id$ and its opposite $-id$.
	Although Sinova also states properties of monotonicity relative to the deepest point and vanishing at infinity, they are formulated in terms of the behaviour of $E[d_{2,\theta}(\mathcal X,A)]$ instead of $A$ itself, following from the definition.
\end{remark}

\subsubsection{Maximality at the center of symmetry}

While $F$-symmetry is suitable for  $D_{r}$ and $RD_{r},$ we will use  $(\midd ,\spr)$-symmetry for  $D_{r}^{\theta}$ and $RD_{r}^{\theta}$ as the $\midd$ and $\spr$ functions are involved in their construction.
We first focus on cases $r=1,2,$ since some of our proofs employ arguments which are specific to those values.

For $r=1,$ the results are in Propositions \ref{resultadoDr2} and \ref{P2Dtheta}, which require integrably bounded fuzzy random variables. These propositions  rely on Lemmas \ref{lemaintegrably} and \ref{lemamidintegrably}, which  ensure the existence of the expectation in the denominator of $D_{1}$ and $D_{1}^{\theta},$ respectively.  Note that for $r=1$ one has $RD_1=D_1$ and $RD_1^\theta=D_1^\theta$.

\begin{lemma}\label{lemaintegrably}\label{hum}\label{rug}\label{grunt}
	Let $r,s\in[1,\infty)$ and $\mathcal{X}\in L^{r}[\mathcal{F}_{c}(\mathbb{R}^{p})].$ Then $\text{E}[\rho_{s}(\text{I}_{\{0\}}, \mathcal{X})^r] < \infty$.
\end{lemma}

\begin{lemma}\label{lemamidintegrably}\label{far}
	Let $r\in[1,\infty),$  $\theta\in[0,\infty)$ and
	 $\mathcal{X}\in L^{1}[\mathcal{F}_{c}(\mathbb{R}^{p})]$. Then $E[d_{r,\theta}(\text{I}_{\{0\}}, \mathcal X)]<\infty.$
\end{lemma}

	\begin{proposition}\label{resultadoDr2}
		Let $\mathcal{J} = \mathcal{F}_{c}(\mathbb{R}^{p})$ and $\mathcal{H}\subseteq L^{1}[\mathcal{F}_{c}(\mathbb{R}^{p})]$. Then $D_1$ (equivalently, $RD_{1}$) based on $\mathcal{J}$ and $\mathcal{H}$ satisfies Property P2 for $F$-symmetry.
	\end{proposition}
	
	\begin{proposition}\label{P2Dtheta}
		Let $\mathcal{J} = \mathcal{F}_{c}(\mathbb{R}^{p})$, $\mathcal{H}\subseteq L^{1}[\mathcal{F}_{c}(\mathbb{R}^{p})]$ and $\theta\in [0,\infty)$. Then $D_1^\theta$ (equivalently, $RD_{1}^\theta$) based on $\mathcal{J}$ and $\mathcal{H}$ satisfies Property P2 for $(\midd,\spr)$-symmetry.
	\end{proposition}

For $r=2$, the results are in Propositions \ref{Simetriar2} and \ref{P2Dtheta2}.

\begin{proposition}\label{Simetriar2}
	Let $\mathcal{J} = \mathcal{F}_{c}(\mathbb{R}^{p})$ and $\mathcal{H}\subseteq L^{2}[\mathcal{F}_{c}(\mathbb{R}^{p})]$. 
	Then, $RD_{2}$ based on $\mathcal{J}$ and $\mathcal{H}$ satisfies Property P2 for $F$-symmetry.
\end{proposition}

\begin{proposition}\label{P2Dtheta2}
	Let $\mathcal{J} = \mathcal{F}_{c}(\mathbb{R}^{p})$, $\mathcal{H}\subseteq L^{2}[\mathcal{F}_{c}(\mathbb{R}^{p})]$ and $\theta\in [0,\infty)$. Then, $RD_{2}^{\theta}$ based on $\mathcal{J}$ and $\mathcal{H}$ satisfies Property P2 for $(\midd,\spr)$-symmetry.
\end{proposition}

Fuzzy sets can be associated with their support functions in the function space $L^{r}(\mathbb{S}^{p-1}\times [0,1],\mathcal{V}_{p}\otimes\nu)$. Thus, it is possible to define a notion of symmetry in the fuzzy setting by using  central symmetry in that function space (see \citep{Simetria}). Notice that this notion does not depend on the choice of $r\in[1,\infty)$.

\begin{definition}
Let $\mathcal{X}$ be a fuzzy random variable, we say that $\mathcal{X}$ is \textit{functionally symmetric} with respect to a fuzzy set $A$ if $s_{\mathcal{X}} -  s_{A}$ is identically distributed as $s_{A}-s_{\mathcal{X}}$.
\end{definition}

\begin{theorem}\label{simetriaLr}
	Let $r\in [1,\infty)$, $\mathcal{J} = \mathcal{F}_{c}(\mathbb{R}^{p})$ and $\mathcal{H}\subseteq L^{1}[\mathcal{F}_{c}(\mathbb{R}^{p})]$. Then, $D_{r}$ based on $\mathcal{J}$ and $\mathcal{H}$ satisfies Property P2 for {functional symmetry}.
\end{theorem}

\subsubsection{Properties P3 and P4}

We will now study properties P3 and P4 for $L^r$-type depths.

Lemma \ref{lemaintegrably} guarantees the finteness of the expectation in the denominator of $D_{r}$  for each $r\in [1,\infty)$ and for every integrable bounded fuzzy random variable. In the case of $RD_{r}$,  we consider fuzzy random variables $\mathcal{X}\in L^{r}[\mathcal{F}_{c}(\mathbb{R}^{p})]$.

\begin{theorem}\label{resultadoDr3}\label{aaa}
	Let $r\in[1,\infty)$, $\mathcal{J} = \mathcal{F}_{c}(\mathbb{R}^{p})$, $ \mathcal{H}_{1}\subseteq L^{1}[\mathcal{F}_{c}(\mathbb{R}^{p})]$ and $ \mathcal{H}_{r}\subseteq L^{r}[\mathcal{F}_{c}(\mathbb{R}^{p})]$. Then $D_{r}$ based on $\mathcal{J}$ and $\mathcal{H}_{1}$, and $RD_{r}$ based on $\mathcal{J}$ and $\mathcal{H}_{r}$ both satisfy
	\begin{itemize}
		\item	P3a and P4a,
		\item P3b  for the $\rho_{s}$ and $d_{s,\theta}$ metrics for any $s\in (1,\infty)$ and $\theta\in (0,\infty)$,
		\item P4b for the $\rho_{s}$ and $d_{s,\theta}$ metrics for any $s\in [1,r]$ and $\theta\in (0,\infty)$.
	\end{itemize}
\end{theorem}

In general, P4b does not admit $s>r$, as shown for $r=1$ in \cite[Example 5.9]{primerarticulo}. Based on Lemma \ref{lemamidintegrably}, the function $D_{r}^{\theta}$ is well defined for integrably bounded fuzzy random variables. For the case of $RD_{r}^{\theta}$, we consider fuzzy random variables $\mathcal{X}\in L^{r}[\mathcal{F}_{c}(\mathbb{R}^{p})]$.

\begin{theorem}\label{P34Dtheta}
	Let $r\in[1,\infty)$, $\theta\in[0,\infty)$, $\mathcal{J} = \mathcal{F}_{c}(\mathbb{R}^{p})$, $ \mathcal{H}_{1}\subseteq L^{1}[\mathcal{F}_{c}(\mathbb{R}^{p})]$ and $ \mathcal{H}_{r}\subseteq L^{r}[\mathcal{F}_{c}(\mathbb{R}^{p})]$. Then $D_{r}^{\theta}$ based on $\mathcal{J}$ and $\mathcal{H}_{1}$ and $RD_{r}^{\theta}$ based on $\mathcal{J}$ and $\mathcal{H}_{r}$ satisfy
	\begin{itemize}
		\item P3a,
		\item P4a if $\theta\in(0,\infty)$,
		\item P3b for the $\rho_{s}$ and $d_{s,\theta}$ metrics for any $s\in (1,\infty)$ and $\theta\in (0,\infty)$,
		\item P4b for the $\rho_{s}$ and $d_{s,\theta}$ metrics for any $s\in [1,r]$ and $\theta\in (0,\infty)$.
	\end{itemize}
\end{theorem}

\begin{remark}\label{re}
	The case $s=1$ is special for property P3b in Theorems \ref{resultadoDr3} and \ref{P34Dtheta} because the space $L^{1}(\mathbb{S}^{p-1}\times [0,1],\mathcal{V}^{p}\otimes\nu)$ is not strictly convex. This results in that P3a and P3b are not necessarily equivalent for $s=1$, that equivalence being used to prove P3b for $s>1.$
\end{remark}

For $\theta = 0$, Properties P4a and P4b are not satisfied, as shown next. Note that the distance function $d_{2,\theta}$ is defined in \cite{Trutsching} for $\theta\in (0,1].$ It is not a distance for $\theta = 0,$ as mentioned in Section \ref{Ld}.

\begin{proposition}\label{P4Dtheta} 
 $D_{r}^{0}$ and $RD_{r}^{0}$ can fail P4a for any $r\in [1,\infty)$ and P4b with the  $\rho_{s}$ metric for any $s\in (1,\infty)$ and $r\in [1,\infty).$
\end{proposition}

\section{Real data example}\label{simulations}
	In order to compare the behavior of projection and $L^{r}$-type depths we use the dataset \textit{Trees} from the {\tt SAFD} (Statistical Analysis of Fuzzy Data) R package \cite{Colubi}. It comes from a reforestation study at the INDUROT forest institute in Spain. The study considers the  \textit{quality} of the tree, a fuzzy random variable whose observations are trapezoidal fuzzy numbers. To define it, experts took into account different aspects of the trees, including leaf structure and height-diameter ratio. 
The $x$-axis represents quality, in a scale from 0 to 5, where 0 means null quality and 5  perfect quality. The $y$-axis represents membership.  The dataset contains a random sample (size: 279) of 9 possible fuzzy values (see Figure \ref{trees} and Table \ref{tabladifusos}). There, the trapezoidal fuzzy numbers are represented by $T_i,$ $i=1, \ldots, 9$ from left to right, for which projection depth and some $L^{r}$-type depths were computed.

	\begin{figure}[htbp]
		\begin{center}
			\includegraphics[width=0.49\linewidth,scale = 0.1]{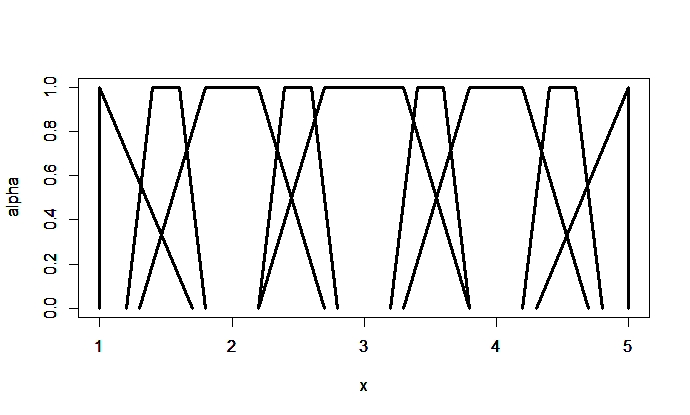}	
		\end{center}
		\caption{Display of the dataset \textit{Trees}.}
		\label{trees}
	\end{figure}
		\begin{table}[h]\label{tabladepth}	
		\begin{center}			
			\begin{tabular}{|l|ccccccccr|}				
				\hline							
				 \text{Quality}& $T_{1}$ & $T_{2}$ & $T_{3}$ & $T_{4}$ & $T_{5}$ & $T_{6}$ & $T_{7}$ & $T_{8}$ & $T_{9}$ \\ 				
				\hline 				
				\text{Frequency}&22 & 16 & 39 & 36 & 85 & 22 & 35 & 12 & 12 \\	
				\hline
				$PD_F$ & .2333 & .2917 & .3889 & .4737 & 1 & .4737 & .3889 & .2917 & .2333\\
				\hline
				$D_1$ & .3726 & .4149 & .4887 & .5488 & .5790 & .5163 & .4493 & .3781 & .3814		\\				
				\hline				
				$D_2$ & .4530 & .4751 & .5214 & .5506 & .5903 & .5287 & .5001 & .4545 & .4564 \\ 
				\hline
				$D_1^5$ & .2979 & .3036 & .3761 & .3695 & .3814 & .3545 & .3307 & .2838 & .2522 \\
				\hline
				$D_1^{10}$ & .2295 & .2390 & .2972 & .2780 & .2839 & .2694 & .2682 & .2265 & .2014 \\ 
				\hline
			\end{tabular}	
		\end{center}		
		\caption{Sample frequencies and depths for each quality value.}	
		\label{tabladifusos}	
	\end{table}

 In Figure \ref{trees}, we appreciate a certain symmetry in the data representation. Beyond this fact, we can not discard any metric \textit{a priori}, thus we compute the $L^{r}$-type dephts for $r = 1,2$, the most common cases in the literature. 
It is clear that the median of the sample (in the sense of \cite{medianbea}) is the maximizer of $D_1$ and thus is $T_{5}$. This fact, together with the fact of symmetry, makes us suppose that the projection depth will give a symmetric ordering, that is $T_{1}$ will have the same depth of $T_{9}$, $T_{2}$ the same depth of $T_{8}$ and so on. Table \ref{tabladifusos} provides that the projection depth represents the symmetry of the data. In the left panel of Figure \ref{Projectiontrees} it is represented the ordering which induce the projection depth. The ordering induced in the fuzzy numbers by the 1- and 2-natural depths is the same and it is represented in  the right panel of Figure \ref{Projectiontrees}.

	\begin{figure}[htbp]
		\begin{center}
			\includegraphics[width=0.49\linewidth,scale = 0.1]{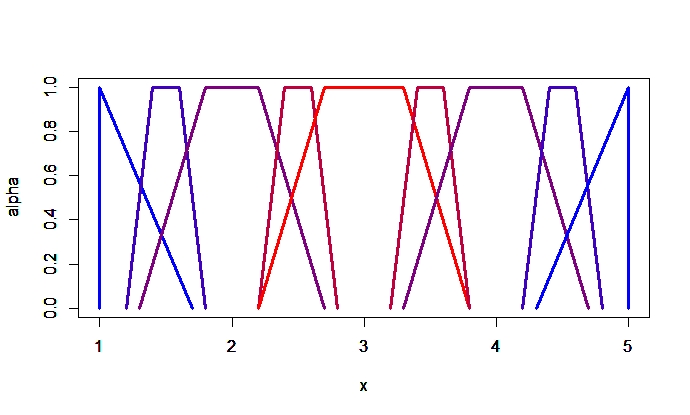}	
					\includegraphics[width=0.49\linewidth,scale = 0.1]{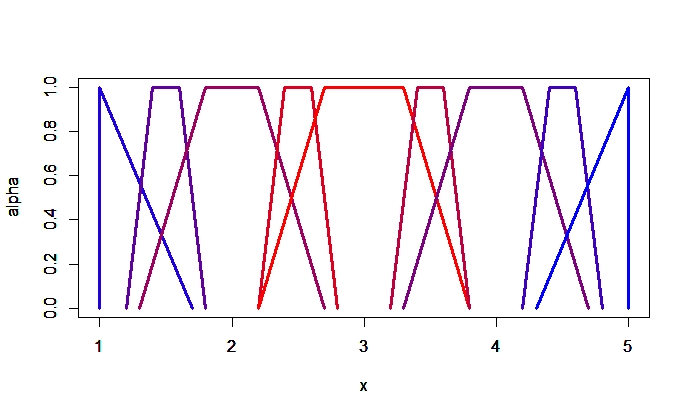}	
		\end{center}
		\caption{Display of the dataset \textit{Trees}. Color is assigned based on the projection depth (left panel) and on the $1$-natural and $2$-natural depths (right panel) of each fuzzy set   in the empirical distribution. Colors range from red (high depth) to blue (low depth).}
		\label{Projectiontrees}
	\end{figure}

	Finally, we compute some examples of $(r,\theta)$-location depth. The cases $r = 1,2$ and $\theta = 1$ generate the same ordering as $D_{1}$ and $D_{2}$. If we take $\theta>1$, we prioritize \textit{shape} over \textit{location} and we should expect a different ordering (see Figure \ref{figuralrtheta}). Indeed, as $\theta$ increases, trapezoidal fuzzy sets with intermediate slopes become deeper than centrally located ones.
	
	\begin{figure}[htbp]
		\begin{center}
			\includegraphics[width=0.49\linewidth,scale = 0.1]{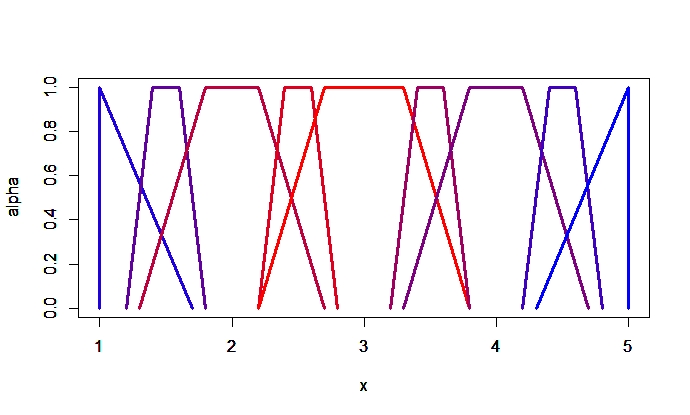}	
			\includegraphics[width=0.49\linewidth,scale = 0.1]{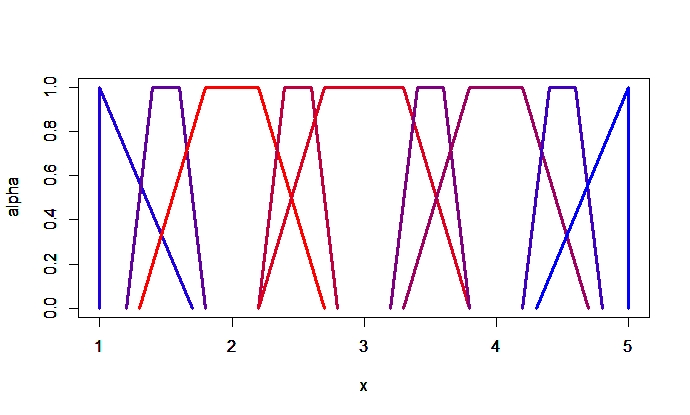}	
		\end{center}
		\caption{Display of the dataset \textit{Trees}. Color is assigned based on the $(1,5)$-location depth (left) and $(1,10)$-location  depth (right), ranging from red (high depth) to blue (low depth).}
		\label{figuralrtheta}
	\end{figure}

\section{Proofs}\label{proofs}
This section contains the mathematical proofs of the results in Sections \ref{[[3]]} and \ref{[[4]]}.

The proof of Theorem \ref{theoremprojectionsemilinear} relies on Lemmas \ref{teorema3b} and \ref{teorema4b} below  \cite[Theorem 5.4 and Proposition 5.8]{primerarticulo}. 	
Given a  metric $d$ in $\mathcal{F}_{c}(\mathbb{R}^{p}),$ these lemmas consider the following assumptions.
	\begin{itemize}
		\item[(A1)] $d(\gamma\cdot A,\gamma\cdot B) = \gamma\cdot d(A,B)$ for all $A,B\in\mathcal{F}_{c}(\mathbb{R}^{p})$ and $\gamma\in [0,\infty),$
		\item[(A2)] $d(A+W,B+W) = d(A,B)$ for all $A,B,W\in\mathcal{F}_{c}(\mathbb{R}^{p}).$
	\end{itemize}

	\begin{lemma}\label{teorema3b}
		Let $(\mathbb{E},\|\cdot\|)$ be a strictly convex Banach space, $d$ a metric in $\mathcal{F}_{c}(\mathbb{R}^{p})$ fulfilling $A1$ and $A2$, and $j: (\mathcal{F}_{c}(\mathbb{R}^{p}),d)\rightarrow (\mathbb{E},\|\cdot\|)$  an isometry. Whenever $A,B,C\in\mathcal{F}_{c}(\mathbb{R}^{p})$ are such that $d(A,B) = d(A,C) + d(B,C)$, the fuzzy set $C$ has the form $(1-\lambda)\cdot A + \lambda\cdot B$ for some $\lambda\in [0,1]$.
	\end{lemma}

	\begin{lemma}\label{teorema4b}\label{[[6.2]]}
		Let $\mathcal{X}$ be a fuzzy random variable and $D(\cdot;\mathcal{X}):\mathcal{F}_{c}(\mathbb{R}^{p})\rightarrow [0,\infty)$ a function for which P4b holds with respect to a metric that fulfills $A1$ and $A2$. Then $D(\cdot;\mathcal{X})$ satisfies P4a.
	\end{lemma}

	We introduce a basic result about symmetry of real random variables which is used in the proof of property P2 below.
	
		\begin{lemma}\label{teoremasimetriaReales}\label{xyz}
			Let $X$ be a real random variable symmetric with respect to $c\in\mathbb{R}$. Then $c = \text{med}(X)$ and also $c = \text{E}[X]$ provided $\text{E}[X]\in\R$ exists.
		\end{lemma}
		
		\begin{proof}
			If $\text{E}[X]< \infty,$  then
			$\text{E}[X] - c = \text{E}[X - c] = \text{E}[c - X] = c - \text{E}[X],$
			where the second equality is due to the symmetric hypothesis ($X - c$ and $c - X$ are equally distributed). Thus, $\text{E}[X] = c$.	
			
			Suppose for a contradiction that $c\not\in\text{Med}(X)$. Without loss of generality we assume  $\mathbb{P}(X\leq c) < 1/2$. Therefore, $\mathbb{P}(X\geq c) > 1/2$. By the symmetry hypothesis,  $\mathbb{P}(X - c\leq 0) = \mathbb{P}(c - X\leq 0).$ Thus, we have that
			$1/2 > \mathbb{P}(X\leq c) = \mathbb{P}(X\geq c) > 1/2,$
			which leads to a contradiction. Then,
			\begin{equation}\label{p0}
				c \in \text{Med}(X).
			\end{equation}
			If we restrict $\text{Med}(X)$ to be a singleton, then $c = \text{med}(X)$.
			If the set $\text{Med}(X)=[m,M]$ is not a singleton, let us assume for a contradiction that $c\neq \text{med}(X) = (M + m)/2$.
			Taking into account \eqref{p0}  we assume, without loss of generality, 
			\begin{equation}\label{pg}
				\cfrac{(M + m)}{2}< c\leq M.
			\end{equation}
			That implies $M - c < c - m.$ Then  there exists some $\epsilon > 0$ such that $c - [(M - c) + \epsilon] > m$.
			As \eqref{pg} also implies 
			\begin{equation}\label{p4}
				M+\epsilon>c,
			\end{equation}
			we get $ c > c - [(M - c) + \epsilon] > m.$ Then $c - [(M - c) + \epsilon]\in \text{Med}(X)$ and 
			\begin{equation}\label{p2}
				\mathbb{P}(X\leq c - [(M - c) + \epsilon])\geq \cfrac{1}{2}.
			\end{equation}
			As $\mathbb{P}(X\leq M + \epsilon)\geq \mathbb{P}(X\leq c)\geq 1/2,$ by \eqref{p0} and  \eqref{p4}, and $M + \epsilon\notin \text{Med}(X),$ 
			\begin{equation}\label{p1}
				\mathbb{P}(X\geq M + \epsilon) < \cfrac{1}{2}.
			\end{equation}
			By the central symmetry of $X$, 			$\mathbb{P}(X - c\leq t) = \mathbb{P}(c - X\leq t)$ for each $t\in\mathbb{R}$. Setting $t = -[(M - c) + \epsilon]$ and taking into account \eqref{p2} and \eqref{p1},
			$$1/2\leq\mathbb{P}(X\leq c - [(M - c) + \epsilon]) = \mathbb{P}(X\geq c + (M - c) + \epsilon) < 1/2,$$
			a contradiction.
		\end{proof}

\begin{proof}[Proof of Theorem \ref{theoremprojectionsemilinear}]

	\emph{Property P1.}
	Let $M\in\mathcal{M}_{p\times p}(\mathbb{R})$ be a regular matrix and $A,B\in\mathcal{F}_{c}(\mathbb{R}^{p})$. It suffices to prove $O(M\cdot A + B; M\cdot\mathcal{X} + B) = O(A;\mathcal{X})$. By translation invariance, 
	$$\text{MAD}(s_{\mathcal{X}}(u,\alpha) + s_{B}(u,\alpha)) = \text{MAD}(s_{\mathcal{X}}(u,\alpha))$$ 
for any $u\in\mathbb{S}^{p-1}$ and $\alpha\in [0,1],$ yielding $O(M\cdot A + B;M\cdot\mathcal{X} + B) = O(M\cdot A;M\cdot\mathcal{X})$.

	Now consider the function $g:\mathbb{S}^{p-1}\rightarrow\mathbb{S}^{p-1}$ defined by 
	$$
	g(u) = \left(\cfrac{1}{\left\|M^{T}\cdot u\right\|}\right)M^{T}\cdot u.
	$$
 	Then
	\begin{equation*}
		O(M\cdot A;M\cdot\mathcal{X}) = \sup_{u\in\mathbb{S}^{p-1},\alpha\in [0,1]}\cfrac{|s_{A}(g(u),\alpha) - \text{med}(s_{\mathcal{X}}(g(u),\alpha))|}{\text{MAD}(s_{\mathcal{X}}(g(u),\alpha))} =
		O(A;\mathcal{X}),
	\end{equation*}
	where  the first identity uses \eqref{pa} and the properties of the univariate median. The second identity holds because $g$ is bijective, a consequence of  $M$ being regular.

	\emph{Property P2.}
	Let $\mathcal{X}$ be a fuzzy random variable, $F$-symmetric with respect to some $A\in\mathcal{F}_{c}(\mathbb{R}^{p})$. It implies that the real random variable $s_{\mathcal X}(u,\alpha)$ is symmetric with respect to $s_A(u,\alpha)$ for every $u\in\mathbb S^{p-1}$ and $\alpha\in [0,1]$. By Lemma \ref{xyz}, $s_{A}(u,\alpha) = \text{med}(s_{\mathcal{X}}(u,\alpha))$ for all $u\in\mathbb{S}^{p-1}$ and $\alpha\in [0,1],$ thus $O(A;\mathcal{X}) = 0.$ As $O(U;\mathcal{X})\geq 0$ for all $U\in\mathcal{F}_{c}(\mathbb{R}^{p})$, we obtain
	$$
	D_{FP}(A;\mathcal{X}) = 1 \geq \sup_{U\in\mathcal{F}_{c}(\mathbb{R}^{p})} D_{FP}(U;\mathcal{X}).
	$$
	
	\emph{Property P3a.}
	It is not hard to show that $O(C;\mathcal{X})$ is a convex function in $C$, i.e. $$O((1-\lambda)\cdot U + \lambda\cdot V;\mathcal{X})\leq (1-\lambda)\cdot O(U;\mathcal{X}) + \lambda\cdot O(V;\mathcal{X})$$ for all $U,V\in\mathcal{F}_{c}(\mathbb{R}^{p})$ and $\lambda\in [0,1],$ using the linearity of the support function, the triangle inequality and the fact that a sum of suprema majorizes the supremum of sums.
	Then, taking any $A,B\in\mathcal{F}_{c}(\mathbb{R}^{p})$  such that $A$ maximizes $D_{FP}(\cdot;\mathcal{X}),$ 
	\begin{equation}
		\begin{aligned}\nonumber
			D_{FP}((1-\lambda)\cdot A + \lambda\cdot B;\mathcal{X}) &= (1 + O((1-\lambda)\cdot A + \lambda\cdot B;\mathcal{X}))^{-1}\geq\\ \nonumber
			&(1 + (1-\lambda)\cdot O(A;\mathcal{X}) + \lambda\cdot O(B;\mathcal{X}))^{-1}\geq\\ \nonumber
			&(1 + O(B;\mathcal{X}))^{-1} = D_{FP}(B;\mathcal{X}).
		\end{aligned}
	\end{equation}

	\emph{Property P3b.} By Lemma \ref{teorema3b}, P3a and P3b are equivalent for all $\rho_{r}$ with $r\in (1,\infty)$.
	
	\emph{Property P4b.} Let  $r\in [1,\infty)$. Let $A\in\mathcal{F}_{c}(\mathbb{R}^{p})$ maximize $D_{FP}(\cdot;\mathcal{X})$ and let $\{A_{n}\}_{n}$ be a sequence of fuzzy sets such that $\lim_{n}\rho_{r}(A,A_{n}) = \infty$.
	As $\rho_{r}(A,A_{n})\leq d_{\infty}(A,A_{n})$ for every $n\in\mathbb{N},$  we have $\lim_{n} d_{\infty}(A,A_{n}) = \infty$. By the triangle inequality,
	\begin{equation}\label{in}
		\lim_{n} d_{\infty}(A_{n},\text{I}_{\{0\}}) = \infty.
	\end{equation}
	Let us denote by $A_{n,\alpha}$ the $\alpha$-level of $A_n.$ As $d_H(A_{n,\alpha},\{0\}) = \sup\{\|x\| : x\in A_{n,\alpha}\}$ and $A_{n,\alpha}\subseteq A_{n,0}$ for all $\alpha\in [0,1]$ and $n\in\mathbb{N}$, we have
	\begin{equation}\label{su}
		d_{\infty}(A_{n},\text{I}_{\{0\}}) = d_H(A_{n,0},\{0\}) = \sup\{\|x\| : x\in A_{n,0}\}.
	\end{equation}
	Since the norm is continuous as a function and each $A_{n,0}$ is compact, the supremum is attained at some $x_n\in A_{n,0}$. Thus 
$$\lim_{n}\|x_{n}\| = \lim_{n} d_{\infty}(A_{n},\text{I}_{\{0\}})=\infty.$$
	In particular, some $e_i$ in the standard basis $\{e_1,\ldots,e_p\}$ of $\R^p$ is such that $\lim_{n} \langle e_i,x_{n}\rangle = \infty$. As $\langle e_i,x_{n}\rangle\leq s_{A_{n}}(e_i,0)$ for every $n\in\mathbb{N}$, we have  $\lim_{n} s_{A_{n}}(e_i,0) = \infty$.
	Taking this into account, since $\mbox{med}(s_{\mathcal{X}}(e_i,0))\in\mathbb{R}$ and $\mbox{MAD}(s_{\mathcal{X}}(e_i,0))\in [0,\infty),$
	\begin{equation*}
		\lim_{n\rightarrow\infty}O(A_{n};\mathcal{X})\geq\lim_{n\rightarrow\infty}\cfrac{\left|s_{A_{n}}(e_i,0) - \mbox{med}(s_{\mathcal{X}}(e_i,0))\right|}{\mbox{MAD}(s_{\mathcal{X}}(e_i,0))} = \infty.
	\end{equation*}
	Then, $\lim_{n}D_{FP}(A_{n};\mathcal{X}) = 0,$ and $D_{FP}$ satisfies P4b for the $\rho_{r}$ metric for every $r\in [1,\infty)$, as well as for $d_\infty$.
	
	Now, let $\{A_{n}\}_{n}$ be a sequence such that $\lim_{n} d_{r}(A,A_{n}) = \infty$ for some $r\in [1,\infty)$. As $d_{r}(A,A_{n})\leq d_{\infty}(A,A_{n})$ for every $n\in\mathbb{N}$, the same proof establishes P4b for $d_{r}$. 
	
	\emph{Property P4a.} By Lemma \ref{teorema4b}, P4b for the $\rho_r$-metric implies P4a, for any $r\in [1,\infty)$.
\end{proof}

\begin{proof}[Proof of Proposition \ref{P3dtheta}]
	\begin{enumerate}
	\item[Case 1 ($\rho_{r}$ and $\rho_{r}^{r}$).] For $r\in [1,\infty)$ and $\mathcal{X}\in L^{r}[\mathcal{F}_{c}(\mathbb{R}^{p})]$,
	\begin{equation}
		\begin{aligned}
			&\rho_{r}((1-\lambda)\cdot A + \lambda\cdot B,\mathcal{X}) = \|s_{(1-\lambda)\cdot A + \lambda\cdot B} - s_{\mathcal{X}}\|_{r} = \\ \nonumber 
			&\|(1-\lambda)\cdot s_{A} +\lambda\cdot s_{B} - s_{\mathcal{X}}\|_{r} = \|(1-\lambda)\cdot (s_{A} - s_{\mathcal{X}}) + \lambda\cdot (s_{B}-s_{\mathcal{X}})\|_{r}\leq \\ \nonumber
			&\|(1-\lambda)\cdot (s_{A} - s_{\mathcal{X}})\|_{r} + \|\lambda\cdot (s_{B}-s_{\mathcal{X}})\|_{r} = (1-\lambda)\cdot \rho_{r}(A,\mathcal{X}) + \lambda\cdot \rho_{r}(B,\mathcal{X}),
		\end{aligned}
	\end{equation}
for every $A,B\in\mathcal{F}_{c}(\mathbb{R}^{p})$ and $\lambda\in [0,1]$, where the inequality is due to the triangle inequality and the second equality due to the linearity of the support function.

Now let us consider the function $f:[0,\infty)\rightarrow [0,\infty)$ defined by $f(x) = x^r$. The function $f$ is convex and increasing, thus
\begin{equation}
	\begin{aligned}
		&\rho_{r}((1-\lambda)\cdot A + \lambda\cdot B,\mathcal{X}) ^r = f\left(\rho_{r}((1-\lambda)\cdot A + \lambda\cdot B,\mathcal{X})\right)\leq\\ \nonumber
		&f\left((1-\lambda)\cdot\rho_{r}(A,\mathcal{X}) + \lambda\cdot\rho_{r}(B,\mathcal{X})\right)\leq (1-\lambda)\cdot f\left(\rho_{r}(A,\mathcal{X})\right) + \lambda\cdot f\left(\rho_{r}(B,\mathcal{X})\right) =\\ \nonumber
		&(1-\lambda)\cdot\rho_{r}(A,\mathcal{X})^r + \lambda\cdot\rho_{r}(B,\mathcal{X})^r,
	\end{aligned}
\end{equation} 
for all $A,B\in\mathcal{F}_{c}(\mathbb{R}^{p})$ and $\lambda\in [0,1]$.

\item[Case 2 ($d_{r,\theta}$ and $d_{r,\theta}^{r}$).] Let $r\in [1,\infty)$, $\theta\in [0,\infty)$ and $\mathcal{X}\in L^{r}[\mathcal{F}_{c}(\mathbb{R}^{p})]$. The mapping 
$$
(\|\cdot\|_r^r +\theta\|\cdot\|_r^r)^{1/r} : L^{r}(\mathbb{S}^{p-1}\times [0,1],\mathcal{V}_{p}\otimes\nu)\oplus_r L^{r}(\mathbb{S}^{p-1}\otimes [0,1],\theta^{1/r}\cdot (\mathcal{V}_{p}\otimes\nu))\rightarrow [0,\infty)
$$ 
is a norm. We identify each $A\in\mathcal{F}_c(\mathbb{R}^p)$ with the pair 
$$
(\midd(s_{A}),\spr(s_{A}))\in L^{r}(\mathbb{S}^{p-1}\times [0,1],\mathcal{V}_{p}\otimes\nu)\oplus_r L^{r}(\mathbb{S}^{p-1}\otimes [0,1],\theta^{1/r}\cdot (\mathcal{V}_{p}\otimes\nu))
$$
Using the properties of $\midd$, $\spr$ and support functions one obtains
$$
h(s_{(1-\lambda)\cdot A}) + h(s_{\lambda\cdot B}) = h(s_{(1-\lambda)\cdot A+\lambda\cdot B})
$$
for every $h\in\{\midd,\spr\}$, $A,B\in\mathcal{F}_c(\mathbb{R}^p)$ and $\lambda\in [0,1]$. Now
\begin{equation}\nonumber
	\begin{aligned}
		&d_{r,\theta}((1-\lambda)\cdot A+\lambda\cdot B;\mathcal{X}) = \\ \nonumber &\Bigl(\|\midd (s_{(1-\lambda)\cdot A+\lambda\cdot B}) - \midd(s_{\mathcal{X}})\|_r^r + \theta\cdot\|\spr (s_{(1-\lambda)\cdot A+\lambda\cdot B}) - spr(s_{\mathcal{X}})\|_r^r\Bigr)^{1/r} = \\ \nonumber
		&\Bigl( \|(1-\lambda)\cdot (\midd(s_{A}) - \midd(s_{\mathcal{X}})) + \lambda\cdot (\midd (s_{B}) - \midd(s_{\mathcal{X}}))\|_r^r +\\ \nonumber
		&\theta\cdot\|(1-\lambda)\cdot (\spr(s_{A}) - \spr(s_{\mathcal{X}})) + \lambda\cdot (\spr(s_{B}) - \spr(s_{\mathcal{X}}))\|_r^r\Bigr)^{1/r}\leq\\ \nonumber
		&\Bigl(\|(1-\lambda)\cdot (\midd(s_{A}) - \midd(s_{\mathcal{X}}))\|_r^r+\theta\cdot\|(1-\lambda)\cdot (\spr(s_{A}) - \spr(s_{\mathcal{X}}))\|_r^r\Bigr)^{1/r} + \\ \nonumber
		&\Bigl(\|\lambda\cdot (\midd(s_{B}) - \midd(s_{\mathcal{X}}))\|_r^r+\theta\cdot\|\lambda\cdot (\spr(s_{B}) - \spr(s_{\mathcal{X}}))\|_r^r\Bigr)^{1/r} = \\ \nonumber
		&(1-\lambda)\cdot d_{r,\theta}(A,\mathcal{X}) + \lambda\cdot d_{r,\theta}(B,\mathcal{X})
	\end{aligned}
\end{equation}
where the inequality is due to the triangle inequality for the norm $(\|\cdot\|_r^r+\theta\|\cdot\|_r^r)^{1/r}$.

The proof for $d_{r,\theta}^r$ is analogous to that of $\rho_r^r$.

\end{enumerate}
\end{proof}

\begin{proof}[Proof of Proposition \ref{resultadoDr1}]
	Let $r\in [1,\infty]$ and let $M, A$ and $B$ be as in P1$\ast$.
	For $D_{r}$ and $RD_{r}$, it suffices to prove that for every $\omega\in\Omega$ we have $\rho_{r}(A, \mathcal{X}(\omega)) = \rho_{r}(M\cdot A, M\cdot\mathcal{X}(\omega)),$ as, by \eqref{rhor}, clearly $\rho_{r}(M\cdot A + B, M\cdot\mathcal{X}(\omega) + B) = \rho_{r}(M\cdot A, M\cdot\mathcal{X}(\omega))$. Since
	\begin{equation*}
		\rho_{r}(M\cdot A, M\cdot\mathcal{X}(\omega)) = \left(\int_{[0,1]}\int_{\mathbb{S}^{p-1}} |s_{M\cdot A}(u,\alpha) - s_{M\cdot\mathcal{X}(\omega)}(u,\alpha)|^{r}\dif\mathcal{V}_{p}(u)\, \dif\nu(\alpha) \right)^{1/r},
	\end{equation*}
	using  \eqref{pa} and the orthogonality of $M$,
	\begin{equation*}
		\rho_{r}(M\cdot A, M\cdot\mathcal{X}(\omega)) =\left(\int_{[0,1]}\int_{\mathbb{S}^{p-1}}\left|s_{A}\left(M^{T}\cdot u,\alpha\right) - s_{\mathcal{X}(\omega)}\left(M^{T}\cdot u,\alpha\right)\right|^{r}\dif\mathcal{V}_{p}(u)\, \dif\nu(\alpha) \right)^{1/r}.
	\end{equation*}
	With the change of variable $v = M^{T}u$ and the notation
	$M = (m_{i,j})_{i,j}$, $u = (u_{1},\ldots,u_{p})$ and $v = (v_{1},\ldots,v_{p})$, we have $u_{i} = \sum_{j = 1}^{p} m_{i,j}\cdot v_{j}$. Thus, the domain of integration remains the same and the Jacobian determinant is $\text{det}(J(Mv)) = \text{det}(M)$. By the orthogonality, $\text{det}(M) = \pm 1$ and $|\text{det}(J(Mv))| = 1$. Thus
	\begin{equation*}
		\rho_{r}(M\cdot A, M\cdot\mathcal{X}(\omega)) = \left(\int_{[0,1]}\int_{\mathbb{S}^{p-1}} |s_{A}(v,\alpha) - s_{\mathcal{X}(\omega)}(v,\alpha)|^{r}\dif\mathcal{V}_{p}(v)\, \dif\nu(\alpha) \right)^{1/r} =\rho_{r}(A,\mathcal{X}(\omega)).
	\end{equation*}
The proof for $D_{r}^{\theta}$  and $RD_{r}^{\theta},$ $\theta\in [0,\infty)$ follows similar ideas, as shown next.
	It suffices to prove that
	\begin{equation}
		\begin{aligned}\nonumber
			\|\midd(s_{M\cdot\mathcal{X}(\omega)}) - \midd(s_{M\cdot A})\|_{r} &= \|\midd(s_{\mathcal{X}(\omega)}) - \midd(s_{A})\|_{r}  \mbox{ and } \\ \nonumber
			\|\spr(s_{M\cdot\mathcal{X}(\omega)}) - \spr(s_{M\cdot A})\|_{r} &= \|\spr(s_{\mathcal{X}(\omega)}) - \spr(s_{A})\|_{r}
		\end{aligned}
	\end{equation}
	for any orthogonal matrix $M\in\mathcal{M}_{p\times p}(\mathbb{R})$ and   $\omega\in\Omega.$
	
	As before, by  \eqref{pa} and the orthogonality of $M$,
	\begin{equation}
		\begin{aligned}\nonumber
	&\|\midd(s_{M\cdot\mathcal{X}(\omega)}) - \midd(s_{M\cdot A})\|_{r} \\ 
	=&\left(\int_{[0,1]}\int_{\mathbb{S}^{p-1}} |\midd(s_{M\cdot\mathcal{X}(\omega)})(u,\alpha) - \midd(s_{M\cdot A})(u,\alpha)|^{r}\dif\mathcal{V}_{p}(u)\, \dif\nu(\alpha) \right)^{1/r} \\
	=&\left(\int_{[0,1]}\int_{\mathbb{S}^{p-1}} \left|\midd(s_{\mathcal{X}(\omega)})\left(M^{T}\cdot u,\alpha\right) - \midd(s_{A})\left(M^{T}\cdot u,\alpha\right)\right|^{r}\dif\mathcal{V}_{p}(u)\, \dif\nu(\alpha) \right)^{1/r}
	\end{aligned}
\end{equation}
	Again, with the change of variable $v = M^{T}\cdot u$ we obtain
	\begin{equation}
	\begin{aligned}\nonumber
	&\left\|\midd(s_{M\cdot\mathcal{X}(\omega)}) - \midd(s_{M\cdot A})\right\|_{r} \\
	=&\left(\int_{[0,1]}\int_{\mathbb{S}^{p-1}}|\midd(s_{\mathcal{X}(\omega)})(v,\alpha) - \midd(s_{A})(v,\alpha)|^{r}\dif\mathcal{V}_{p}(v)\, \dif\nu(\alpha) \right)^{1/r} \\
	=&\|\midd(s_{\mathcal{X}(\omega)}) - \midd(s_{A})\|_{r}
\end{aligned}
\end{equation}

	The proof for the spread function is analogous.
\end{proof}

\begin{proof}[Proof of Lemma \ref{lemaintegrably}]
	For any $\omega\in\Omega$ and $\alpha\in [0,1]$ we have $\mathcal{X}_{\alpha}(\omega)\subseteq\mathcal{X}_{0}(\omega),$ which implies  $|s_{\mathcal{X}(\omega)}(u,0)|\geq |s_{\mathcal{X}(\omega)}(u,\alpha)|$ for each $u\in\mathbb{S}^{p-1}.$ Thus
	\begin{equation}\nonumber
		\|\mathcal{X}_{0}(\omega)\|^r = \sup_{u}|s_{\mathcal{X}(\omega)}(u,0)|^r = \sup_{u,\alpha}|s_{\mathcal{X}(\omega)}(u,\alpha)|^r\geq \rho_{r}(\mathcal{X}(\omega),\text{I}_{\{0\}})^r.
	\end{equation}
The inequality holds because the integrand in the definition of $\rho_{s}(\mathcal{X}(\omega),\text{I}_{\{0\}})$ is precisely $|s_{\mathcal{X}(\omega)}(u,\alpha)|$.
	Taking expectations in both sides,
	$$\text{E}[ \rho_{s}(\text{I}_{\{0\}},\mathcal{X}(\omega))^r]\leq\text{E}[\|\mathcal{X}_{0}\|^r] < \infty$$
	because $\mathcal{X}\in L^r[\pfc(\R^p)]$.
\end{proof}

\begin{proof}[Proof of Lemma \ref{lemamidintegrably}]
	Fix  $\theta\in [0,\infty)$ and $r\in [1,\infty).$
	It suffices to prove $\text{E}[d_{r,\theta}(\mathcal{X},\text{I}_{\{0\}})] < \infty.$
	By  \cite[Proposition 4.2]{Trutsching},
	\begin{equation*}
		\left( \int_{\mathbb{S}^{p-1}} |\midd(s_{A})(u,\alpha)|^{r} + \theta\cdot |\spr(s_{A})(u,\alpha)|^{r}\dif\mathcal{V}_{p}(u)\right)^{1/r}\leq d_H(A_{\alpha},\{0\})\le\|A_0\|,
	\end{equation*}
	for any $A\in\mathcal{F}_{c}(\mathbb{R}^{p})$ and $\alpha\in [0,1].$  From this and  \eqref{dr}	we obtain
	\begin{equation}
		\begin{aligned}\nonumber
			&\text{E}[d_{r,\theta}(\text{I}_{\{0\}},\mathcal{X})] \leq
			&\text{E}\left[\left(\int_{[0,1]}\|\mathcal X_0\|^{r} \dif\nu(\alpha)\right)^{1/r}\right] = \text{E}[\| \mathcal{X}_0\|]<\infty
		\end{aligned}
	\end{equation}
because $\mathcal{X}$ is integrably bounded.
\end{proof}

\begin{proof}[Proof of Proposition \ref{resultadoDr2}]
	Let $\mathcal{X}\in\mathcal H$ be $F$-symmetric with respect to $A\in\mathcal{F}_{c}(\mathbb{R}^{p})$. As stated in \eqref{Amedian}, $s_{A}(u,\alpha) \in \text{Med}(s_{\mathcal{X}}(u,\alpha))$ for all $u\in\mathbb{S}^{p-1}$ and $\alpha\in [0,1]$.
	Because of that and since the medians of the integrable random variable $s_{\mathcal{X}}(u,\alpha)$ minimize the expected absolute deviation,
	\begin{equation}\label{mi}
		s_{A}(u,\alpha)\in \argmin_{x\in\mathbb{R}}\text{E}(|s_{\mathcal{X}}(u,\alpha) - x|)
	\end{equation}
	for each $u\in\mathbb{S}^{p-1}$ and $\alpha\in [0,1]$.
	Consider $E[\rho_1(U,\mathcal X)]$ for any fixed $U\in\pfc(\R^p)$. Since the function $s_{\mathcal X}$ is jointly measurable in its three arguments \cite[Lemma 4]{Kra}, by  Fubini's theorem and \eqref{rhor}
	\begin{equation*}
		\text{E}[\rho_{1}(U,\mathcal{X})] =  \int_{[0,1]}\int_{\mathbb{S}^{p-1}} \text{E}[|s_{\mathcal{X}}(u,\alpha) - s_{U}(u,\alpha)|]\dif\mathcal{V}_{p}(u)\, \dif\nu(\alpha),
	\end{equation*}

Applying \eqref{mi} now,
	\begin{equation}
		\begin{aligned}\nonumber
			\text{E}[\rho_{1}(U,\mathcal{X})]\textbf{}&\geq\int_{[0,1]}\int_{\mathbb{S}^{p-1}} \text{E}[|s_{\mathcal{X}}(u,\alpha) - s_{A}(u,\alpha)|]\dif\mathcal{V}_{p}(u)\, \dif\nu(\alpha)
			= \text{E}[\rho_{1}(\mathcal{X},A)].
		\end{aligned}
	\end{equation}
	Then $D_{1}(U;\mathcal{X})\leq D_{1}(A;\mathcal{X})$. By the arbitrariness of $U$, property P2 is satisfied.
\end{proof}

\begin{proof}[Proof of Proposition \ref{P2Dtheta}]
	Let $\mathcal{X}\in\mathcal H$ be $(\midd,\spr)$-symmetric with respect to $A\in\mathcal{F}_{c}(\mathbb{R}^{p})$. Applying the same reasoning in the proof of Proposition \ref{resultadoDr2}, but using \eqref{lemamedian} instead of \eqref{Amedian}, to the $\midd$ and $\spr$ functions  separately, we obtain
	$D_{1}^{\theta}(A;\mathcal{X})\geq D_{1}^{\theta}(U;\mathcal{X})$ for all $U\in\mathcal{F}_{c}(\mathbb{R}^{p})$ and $\theta\in [0,\infty)$.
\end{proof}

\begin{proof}[Proof of Proposition \ref{Simetriar2}]
	Let $\mathcal X\in \mathcal H$ be $F$-symmetric with respect to some $A\in\mathcal{F}_{c}(\mathbb{R}^{p})$. This implies $\text{E}[\|\mathcal{X}_{0}\|] < \infty$ and hence $\text{E}[s_{\mathcal{X}}(u,\alpha)] < \infty$ for all $u\in\mathbb{S}^{p-1}$ and $\alpha\in [0,1]$. By the definition of $F$-symmetry, the real random variable $s_{\mathcal{X}}(u,\alpha)$ is centrally symmetric with respect to $s_{A}(u,\alpha)$ for all $u\in\mathbb{S}^{p-1}$ and $\alpha\in [0,1]$. By Lemma \ref{teoremasimetriaReales},  $s_{A}(u,\alpha) = \text{E}[s_{\mathcal{X}}(u,\alpha)]$ for all $u\in\mathbb{S}^{p-1}$ and $\alpha\in [0,1]$.
	For any square integrable random variable,
	$\text{E}[X]=\argmin_{y\in\mathbb{R}}\text{E}[|X - y|^{2}].$ Then, since $\mathcal X\in L^2[\pfc(\R^p)]$,
	\begin{equation}\label{ecuacionminimo}
		s_{A}(u,\alpha)= \argmin_{U\in\mathcal{F}_{c}(\mathbb{R}^{p})} \text{E}[|s_{\mathcal{X}}(u,\alpha) - s_{U}(u,\alpha)|^{2}]
	\end{equation}
	for each $u\in\mathbb{S}^{p-1}$ and $\alpha\in [0,1]$. Like in Proposition \ref{resultadoDr2}, applying Fubini's theorem and \eqref{ecuacionminimo}, we obtain $RD_{2}(U;\mathcal{X}) \leq RD_{2}(A;\mathcal{X})$ for all $U\in\mathcal{F}_{c}(\mathbb{R}^{p})$. Thus $RD_2$ satisfies P2. 
\end{proof}

\begin{proof}[Proof of Proposition \ref{P2Dtheta2}]
	Let $\theta\in [0,\infty)$ and let $\mathcal{X}\in \mathcal H$ be $(\midd,\spr)$-symmetric with respect to $A\in\mathcal{F}_{c}(\mathbb{R}^{p})$. By applying the same reasoning in the proof of Proposition \ref{Simetriar2} but taking into account $\midd(s_{A})(u,\alpha) = \text{E}[\midd(s_{\mathcal{X}})(u,\alpha)]$ and $\spr(s_{A})(u,\alpha) = \text{E}[\spr(s_{\mathcal{X}})(u,\alpha)]$ for every $u\in\mathbb{S}^{p-1}$ and $\alpha\in [0,1]$, one obtains $RD_{2}^{\theta}(A;\mathcal{X})\geq RD_{2}^{\theta}(U;\mathcal{X})$ for all $U\in\mathcal{F}_{c}(\mathbb{R}^{p})$ and $\theta\in [0,\infty)$. 
\end{proof}

\begin{proof}[Proof of Theorem \ref{simetriaLr}]
	Let $\mathcal{X}\in \mathcal H$ be functionally symmetric with respect to $A\in\mathcal{F}_{c}(\mathbb{R}^{p})$ and $r\in [1,\infty)$. By Lemma \ref{lemaintegrably}, $D_{r}(\cdot;\mathcal{X})$ is well defined.
	 To reach 
	$$
	D_{r}(A;\mathcal{X})\geq\sup_{U\in\mathcal{F}_{c}(\mathbb{R}^{p})}D_{r}(U;\mathcal{X})
	$$
	it suffices to prove
	\begin{equation}\label{ecuacionminimoLr}
	\text{E}[\|s_{\mathcal{X}} - s_{A}\|_{r}]\leq\inf_{U\in\mathcal{F}_{c}(\mathbb{R}^{p})}\text{E}[\|s_{\mathcal{X}} - s_{U}\|_{r}].
\end{equation}

Let us denote by $\mathcal{B}$ the Banach space $L^{r}(\mathbb{S}^{p-1}\times [0,1],\mathcal{V}_{p}\otimes\nu)$. As the norm is a convex function, for every $f\in\mathcal{B}$ 
	\begin{equation}\label{ecuacionLrsimetria}
	\text{E}[\|s_{\mathcal{X}} - s_{A}\|_{r}]\leq\cfrac{1}{2}\cdot\text{E}[\|s_{\mathcal{X}} - s_{A} +f\|_{r}] + \cfrac{1}{2}\cdot\text{E}[\|s_{\mathcal{X}} - s_{A} - f\|_{r}].
	\end{equation}

	Since $\mathcal{X}$ is functionally symmetric with respect to $A$, we know $s_{\mathcal{X}}-s_{A}$ and $s_{A}-s_{\mathcal{X}}$ are identically distributed. 
Thence $\|s_{\mathcal X}-s_A+f\|_r$ and $\|-s_{\mathcal X}+s_A+f\|_r$ are identically distributed and the right-hand side of  \eqref{ecuacionLrsimetria} equals
\begin{equation}
	\begin{aligned}\nonumber
	    &\cfrac{1}{2}\cdot\text{E}[\|-(s_{\mathcal{X}} - s_{A} - f)\|_{r}] + \cfrac{1}{2}\cdot\text{E}[\|s_{\mathcal{X}} - s_{A} - f\|_{r}] = 
	    &\text{E}[\|s_{\mathcal{X}} - s_{A} - f\|_{r}].
\end{aligned}
\end{equation}
Therefore
$$
\text{E}[\|s_{\mathcal{X}} - s_{A}\|_{r}]\leq\text{E}[\|s_{\mathcal{X}} - s_{A} - f\|_{r}]
$$
for each $f\in\mathcal{B}$ and
$$
\text{E}[\|s_{\mathcal{X}} - s_{A}\|_{r}]  \leq \inf_{g\in\mathcal{B}}\text{E}[\|s_{\mathcal{X}} - g\|_{r}]\le \inf_{U\in\mathcal{F}_{c}(\mathbb{R}^{p})}\text{E}[\|s_{\mathcal{X}} - s_{U}\|_{r}]
$$
taking all possible $g=s_{A} + f\in \mathcal{B}$  and using the inclusion $\{s_{U} : U\in\mathcal{F}_{c}(\mathbb{R}^{p}) \}\subseteq\mathcal{B}$.
\end{proof}

\begin{proof}[Proof of Theorem \ref{resultadoDr3}]
	Let $r\in [1,\infty)$, $\mathcal{J} = \mathcal{F}_c (\mathbb{R}^p)$, $\mathcal{H}_{1}\subseteq L^1[\mathcal{F}_c (\mathbb{R}^p)]$ and $\mathcal{H}_r\subseteq L^r[\mathcal{F}_c (\mathbb{R}^p)]$.

	\emph{Property P3a.} By Proposition \ref{P3dtheta}, the mappings $d_r(\cdot,\cdot)$ and $d_r^r(\cdot,\cdot)$ are convex in their first argument.  Lemma \ref{teoremaZuo} yields $D_{r}$ based on $\mathcal{J}$ and $\mathcal{H}_1$, as well as $RD_r$ based on $\mathcal{J}$ and $\mathcal{H}_r$, satisfy P3a. Notice Lemma \ref{grunt} ensures that the integrability assumption in Lemma \ref{[[4.6]]} is satisfied, for the classes $\mathcal H_1$ and $\mathcal H_r$ in the statement.

	\emph{Property P3b.} By Lemma \ref{teorema3b}, P3a and P3b are equivalent for the $\rho_{s}$ metric for every $s\in (1,\infty)$.

		Now, for the $d_{s,\theta}$-metrics we want to apply Lemma \ref{teorema3b} as well, with $s\in (1,\infty)$ and $\theta\in (0,\infty)$. The mapping
		$$
		j : \mathcal{F}_c (\mathbb{R}^p)\rightarrow L^s(\mathbb{S}^{p-1}\otimes [0,1],\mathcal{V}_p\otimes\nu)\oplus_s L^s(\mathbb{S}^{p-1}\otimes [0,1],\theta^{(1/r)}\cdot(\mathcal{V}_p\otimes\nu))
		$$
		defined by $j(A) = (\midd(s_A),\spr(s_A))$ is an isometry, considering in $\mathcal{F}_c (\mathbb{R}^p)$ the metric $d_{s,\theta}$ and in $L^s(\mathbb{S}^{p-1}\otimes [0,1],\mathcal{V}_p\otimes\nu)\oplus_s L^s(\mathbb{S}^{p-1}\otimes [0,1],\theta^{(1/s)}\cdot(\mathcal{V}_p\otimes\nu))$ the distance induced by its norm $(\|\cdot\|_s^s+\theta\cdot\|\cdot\|_s^s)^{1/s}$.
		It is clear from its definition that $d_{s,\theta}$ fulfils A1 and A2.
		In order to use the lemma, we need to show that the Banach space 
		$$
		\left(L^s(\mathbb{S}^{p-1}\otimes [0,1],\mathcal{V}_p\otimes\nu)\oplus_s L^s(\mathbb{S}^{p-1}\otimes [0,1],\theta^{(1/s)}\cdot(\mathcal{V}_p\otimes\nu)),(\|\cdot\|_s^s+\theta\cdot\|\cdot\|_s^s)^{1/s}\right)
		$$ 
		is strictly convex.
		
		Let us define the mapping $\psi:[0,1]\rightarrow [0,1]$ with
		$$
		\psi (t) = \left((1-t)^s + \theta\cdot t^s\right)^{1/s}.
		$$	
		It is easy to show 
		$$
		\left(\|f\|_s^s+\theta\cdot\|g\|_s^s\right)^{1/s} = \left(\|f\|_s+\|g\|_s\right)\cdot\psi\left(\cfrac{\|g\|_s}{\|f\|_s + \|g\|_s}\right)
		$$
for every $(f,g)\in L^s(\mathbb{S}^{p-1}\otimes [0,1],\mathcal{V}_p\otimes\nu)\oplus_s L^s(\mathbb{S}^{p-1}\otimes [0,1],\theta^{(1/s)}\cdot(\mathcal{V}_p\otimes\nu))$. By \citep[Theorem 6]{strictdirectsum}, the Banach space 	$L^s(\mathbb{S}^{p-1}\otimes [0,1],\mathcal{V}_p\otimes\nu)\oplus_s L^s(\mathbb{S}^{p-1}\otimes [0,1],\theta^{(1/s)}\cdot(\mathcal{V}_p\otimes\nu))$ will be strictly convex if and only if $L^s(\mathbb{S}^{p-1}\otimes [0,1],\mathcal{V}_p\otimes\nu)$ and $L^s(\mathbb{S}^{p-1}\otimes [0,1],\theta^{1/s}\cdot(\mathcal{V}_p\otimes\nu))$ are strictly convex and the function $\psi$ is strictly convex. For $s\in(1,\infty)$,  $L^s$-spaces are always strictly convex (e.g., \cite[p. 114]{Car}), and $\Psi$ is strictly convex as $\Psi''(t)>0$ for $t\in(0,1)$. Therefore, by Lemma \ref{teorema3b}, P3b for $d_{s,\theta}$ is equivalent to P3a, which has already been established.

	\emph{Property P4b.} Let $\mathcal{X}\in\mathcal{H}_1$ be a fuzzy random variable and $A\in\mathcal{J}$ a fuzzy set maximizing $D_{r}(\cdot;\mathcal{X})$. Let us first prove the case $s=r.$ Let $\{A_{n}\}_{n}$ be a sequence of fuzzy sets in $\mathcal{J}$ such that
	\begin{equation}\label{l}\lim_{n}\rho_{r}(A_n,A) = \infty.\end{equation}
	
	As $r\geq 1,$ by Lemma \ref{rug} $\text{E}[\rho_{r}(I_{\{0\}},\mathcal{X})]$ is finite. As $\rho_{r}(I_{\{0\}},A)$ is a constant, applying  the triangle inequality to $\rho_r(A,\mathcal{X}),$ we obtain
	\begin{equation}
		\label{f}\text{E}[\rho_{r}(A,\mathcal{X})] < \infty.
	\end{equation}
	Using again the triangle inequality,
	\begin{equation}\label{f2}
		\text{E}[\rho_{r}(A_{n},\mathcal{X})]\geq\text{E}[\rho_{r}(A_{n},A) - \rho_{r}(A,\mathcal{X})]=\rho_{r}(A_{n},A) -\text{E}[ \rho_{r}(A,\mathcal{X})] \to\infty,
	\end{equation}
	where the limit is obtained from \eqref{l} and \eqref{f}.
	Accordingly, $D_{r}(A_n,\mathcal X)\to 0$. 

For the general case, notice $\rho_s\le \rho_r$ whenever $s\le r.$ Thus, $\rho_s(A_n,A)\to \infty$ implies $\rho_r(A_n,A)\to \infty$ and therefore $D_r(A_n;\mathcal X)\to 0$ by the former case.

That establishes the result for $D_r$ under the $\rho_s$-metrics. Let us prove it now for $RD_r$.

Let $\mathcal{X}\in\mathcal{H}_r$. Like before, we will prove first the case  $s = r.$ By Jensen's inequality,
\begin{equation}
	\text{E}[\rho_{r}(A_{n},\mathcal{X})^{r}]\geq \text{E}[\rho_{r}(A_{n},\mathcal{X})]^{r}. 
\end{equation}
From \eqref{f2},
\begin{equation*}
	\lim_{n\rightarrow\infty}\text{E}[\rho_{r}(A_{n},\mathcal{X})^{r}] = \infty.
\end{equation*}
Consequently, $RD_{r}(A_{n},\mathcal{X})\rightarrow 0$. The general case  follows as with $D_r.$

Now let us consider the  $d_{s,\theta}$-metrics. Let $s = r$ and  $\theta\in (0,\infty)$. Given a fuzzy random variable $\mathcal{X}\in\mathcal{H}_1$, a fuzzy set $A\in\mathcal{J}$ maximizing $D_r(\cdot;\mathcal{X})$ and a sequence $\{A_n\}_n$  in $\mathcal{J}$ such that
\begin{equation}\label{ecuaciondrtheta1}
\lim_{n}d_{s,\theta}(A_n,A) = \infty.
\end{equation}

By Lemma \ref{lemamidintegrably}, $\text{E}[d_{s,\theta}(\text{I}_{\{0\}},\mathcal{X})] < \infty$. By  \eqref{ecuaciondrtheta1}, $\lim_{n}d_{s,\theta}(A_n,A)^r = \infty$, 
whence
$$
\lim_{n}\|\midd(s_{A_n}) - \midd(s_A)\|_s^r = \infty
$$
or
$$
\lim_{n}\|\spr(s_{A_n}) - \spr(s_A)\|_s^r = \infty
$$
Since the other case is analogous, we assume without loss of generality $\|\midd(s_{A_n})-\midd(s_{A})\|_s^r\rightarrow\infty$. Moreover,
\begin{equation}
	\begin{aligned}\nonumber
		&\|\midd(s_{A_n})-\midd(s_{A})\|_s \\ \nonumber
		=&\left(\int_{[0,1]}\int_{\mathbb{S}^{p-1}} |\midd(s_{A_n})(u,\alpha) - \midd(s_{A})(u,\alpha)|^s \dif\mathcal{V}_p (u)\, \dif\nu(\alpha)\right)^{1/s} \\ 
		=&\cfrac{1}{2}\cdot\left(\int_{[0,1]}\int_{\mathbb{S}^{p-1}}|(s_{A_n}(u,\alpha) - s_{A}(u,\alpha)) + (s_{A}(-u,\alpha) - s_{A_n}(-u,\alpha))|^s \dif\mathcal{V}_p (u)\, \dif\nu(\alpha)\right)^{1/s}\\
		\le &\cfrac{1}{2}\cdot\left(\|s_{A_n} - s_{A}\|_s + \|s_{A_n} - s_{A}\|_s \right) = \rho_s(A,A_n)\le \rho_r(A,A_n)
	\end{aligned}
\end{equation}
whence $\lim_{n}\rho_r(A_n,A) = \infty$. Thus, using the previous proof, the depth function $D_r$ based on $\mathcal{J}$ and $\mathcal{H}_1$  fulfils P4b for $d_{s,\theta}$. 

The case of $RD_r$ based on $\mathcal{J}$ and $\mathcal{H}_{r}$ is done in an analogous way as in the case of $\rho_s$.

	\emph{Property P4a.} As $\rho_{r}$ and $d_{r,\theta}$ metrics fulfil assumptions $A1$ and $A2$, property P4b implies P4a (Lemma \ref{teorema4b}).
\end{proof}

\begin{proof}[Proof of Theorem \ref{P34Dtheta}]
	\
	
	\emph{Property P3a.} Like in the proof of Property P3a in Theorem \ref{resultadoDr3}, the mapping $(\|\cdot\|_{r}^{r}+\theta\cdot\|\cdot\|_{r}^{r})^{1/r}$ is convex (because it is a norm) and, by Lemma \ref{teoremaZuo}, $D_{r}^{\theta}$ and $RD_{r}^{\theta}$ satisfy P3a for any $r\in [1,\infty)$ and $\theta\in [0,\infty)$.
	
	\emph{Property P3b.} By Lemma \ref{teorema3b}, P3b is equivalent to P3a for the $\rho_{s}$ metric if $s\in (1,\infty)$. In the proof of Theorem \ref{resultadoDr3} it was shown that P3b is equivalent to P3a for the $d_{s,\theta}$ metric.

	\emph{Property P4b.} 
	Let $\theta\in (0,\infty)$. Let $\mathcal{X}\in\mathcal H_1$ and let $A\in\mathcal{F}_{c}(\mathbb{R}^{p})$ be a fuzzy set that maximizes $D_{r}^{\theta}(\cdot,\mathcal{X})$. We consider a sequence $\{A_{n}\}_{n}$ of fuzzy sets such that $\rho_r(A_n,A)\to\infty$. 
	By the triangle inequality, for any $h\in\{\midd,\spr\}$,
	\begin{equation}\label{ecuacion1Drtheta}
		\begin{aligned}
			\text{E}[\|h(s_{\mathcal{X}}) - h(s_{A_{n}})\|_{r}] &\ge E[\|h(s_{A_n})-h(s_A)\|_r-\|h(s_{\mathcal X})-h(s_A)\|_r]\\
			&=\|h(s_{A}) - h(s_{A_{n}})\|_{r} - \text{E}[\|h(s_{\mathcal{X}}) - h(s_{A})\|_{r}].
		\end{aligned}
	\end{equation}
	
	On the other hand, as $\lim_{n} \rho_{r}(A,A_{n}) = \infty$ and $\rho_{r}$ is a metric, the triangle inequality yields $\lim_{n} \rho_{r}(A_{n},\text{I}_{\{0\}}) = \infty$. By the decomposition given in  \eqref{sumasoporte},
	$$	\rho_{r}(A_{n},\text{I}_{\{0\}}) = \left(\int_{[0,1]}\int_{\mathbb{S}^{p-1}} |\midd(s_{A_{n}})(u,\alpha) + \spr(s_{A_{n}})(u,\alpha)|^{r}\dif\mathcal{V}_{p}(u)\, \dif\nu(\alpha) \right)^{1/r}$$
	$$=\|\midd(s_{A_n})+\spr(s_{A_n})\|_r\le \|\midd(s_{A_n})\|_r+\|\spr(s_{A_n})\|_r.$$
	Therefore $\lim_{n}\|\midd(s_{A_{n}})\|_{r} = \infty$ and/or $\lim_{n}\|\spr(s_{A_{n}})\|_{r} = \infty$. Since the other case is analogous, without loss of generality assume 
	\begin{equation}\label{lim}\lim_{n}\|\midd(s_{A_{n}})\|_r = \infty.
		\end{equation}
	 Because $\mathcal{X}$ is integrably bounded, by Lemma \ref{far} we have $\text{E}[d_{r,\theta}(\mathcal{X},\text{I}_{\{0\}})] < \infty,$ which implies 
	 \begin{equation}\label{lim2}
	 \text{E}[\|\midd(s_{\mathcal{X}}) - \midd(s_{A})\|_{r}] < \infty.
	 \end{equation}

Then 
	\begin{equation}\label{lim3}
		\begin{aligned}
			&E[d_{r,\theta}(A_n ,\mathcal X)]\ge \text{E}[\|\midd(s_{\mathcal{X}}) - \midd(s_{A_{n}})\|_{r}] \\ 
			&\ge \|\midd(s_{A}) - \midd(s_{A_{n}})\|_{r} - \text{E}[\|\midd(s_{\mathcal{X}}) - \midd(s_{A})\|_{r}]\\
            & \ge \|\midd(s_{A_n})\|_r - \|\midd(s_{A})\|_{r} - \text{E}[\|\midd(s_{\mathcal{X}}) - \midd(s_{A})\|_{r}]\to \infty,
		\end{aligned}
	\end{equation}
	where the first inequality is due to \eqref{dr}, the second one to  \eqref{ecuacion1Drtheta} and the limit to  \eqref{lim} and \eqref{lim2}.
	Consequently, $D_r^\theta(A_n;\mathcal X)\to 0$. That proves the case $s=r$. The case $s<r$ follows like in the  proof of  Theorem \ref{aaa}.
	
Let us prove it now for $RD_r^\theta$ and the $\rho_s$-metrics.

		Let $\theta\in (0,\infty)$. Let $\mathcal{X}\in\mathcal H_r$ and let $A\in\mathcal{F}_{c}(\mathbb{R}^{p})$  maximize $RD_{r}^{\theta}(\cdot,\mathcal{X})$. Let $\{A_{n}\}_{n}$ be a sequence of fuzzy sets such that $\rho_r(A_n,A)\to\infty$. 	By Jensen's inequality,
		\begin{equation*}
		\text{E}[d_{r,\theta}(A_n ,\mathcal X)^{r}]\geq\text{E}[d_{r,\theta}(A_n ,\mathcal X)]^{r}.
	\end{equation*}
By \eqref{lim3}, 
		\begin{equation*}
	\lim_{n\rightarrow\infty}\text{E}[d_{r,\theta}(A_n ,\mathcal X)^{r}] = \infty.
\end{equation*}
Thus $RD_{r}^{\theta}(A_{n},\mathcal{X})\rightarrow 0$. That establishes the case $s=r$. The case $s < r$ is  deduced like in the  proof of Theorem \ref{aaa}.

	The proof of P4b for $D_r^\theta$ and $RD_r^\theta$ with $d_{s,\theta}$ is analogous to that of P4b for $D_r$ and $RD_r$ with respect to the $\rho_s$-metrics (see Theorem \ref{resultadoDr3}), taking into account the inequality $d_{s,\theta}\leq d_{r,\theta}$ for $s\in[1,r]$.

	\emph{Property P4a.} By Lemma \ref{teorema4b}, property P4b for $\rho_r$ implies P4a.
\end{proof}

\begin{proof}[Proof of Proposition \ref{P4Dtheta}]
	Let $(\Omega,\mathcal{A},\mathbb{P})$ be a probabilistic space such that $\Omega = \{\omega_{1}\}$, $\mathcal{A} = \mathcal{P}(\Omega)$ and let $r\in [1,\infty)$. We consider the fuzzy random variable $\mathcal{X}$ defined by $\mathcal{X}(\omega_{1}) := \text{I}_{[-1,1]}$. Let $A = \mathcal{X}(\omega_{1})$ and $A_{n} := \text{I}_{[-n,n]}$ for all $n\in\mathbb{N}$. It is clear that $A$ maximizes $D_{r}^{0}(\cdot;\mathcal{X})$ with $D_{r}^{0}(A;\mathcal{X})=1$, and that
	$\midd(s_{B})(u,\alpha) =0$ for $B\in\{A, A_n\},$
	$\spr(s_{A})(u,\alpha) = 1,$ and $\spr(s_{A_{n}})(u,\alpha) = n$ for all $u\in\mathbb{S}^{0}, \alpha\in [0,1]$ and $n\in\mathbb{N}$.	By the $\midd/\spr$ decomposition \eqref{sumasoporte}, 
	$$
	\lim_{n\rightarrow\infty} \rho_{r}(A_{n},A) = \lim_{n\rightarrow\infty} (\int_{[0,1]}|n-1|^r d\alpha)^{1/r} = \lim_{n\rightarrow\infty} |n-1| = \infty.
	$$
	Taking into account 
	$
	\text{E}[d_{r,0}(A_{n},\mathcal{X})] = 0
	$ for all $n\in\mathbb{N},$ whence 
	$D_{r}^{0}(A_{n};\mathcal{X}) = 1$, i.e., $D_r^0$ fails P4b for $\rho_r$. In the case $r=1$, we have $RD_{1}^{0}(A_{n};\mathcal{X})=D_{1}^{0}(A_{n};\mathcal{X})$ so $RD_r^0$ can fail P4b as well.

	To prove that $D_{r}^{0}$ and $RD_{r}^{0}$ violate P4a, we  use $B := \text{I}_{[-2,2]}$. Let $r\in [1,\infty)$. Note  $A + nB = \text{I}_{[-1-2n,2n+1]}$ for all $n\in\mathbb{N}$. Clearly, $$\midd(s_{\mathcal{X}(\omega_{1})})(u,\alpha) = 0 = \midd(s_{A+nB})(u,\alpha)$$
for all $u\in\mathbb{S}^{0}$ and $\alpha\in [0,1]$. Thus $\text{E}[d_{r,0}(A+nB,\mathcal{X})] = 0$ and 
$$D_{r}^{0}(A + nB;\mathcal{X}) = 1 = RD_{r}^{0}(A+nB;\mathcal{X})$$ 
for all $n\in\mathbb{N}$ whence $D_{r}^{0}$ and $RD_{r}^{0}$ violate P4a.
	
	{\it A fortiori}, by Lemma \ref{[[6.2]]}, this is also a counterexample to property P4b for $\rho_r$, for any $r\in(1,\infty)$.
\end{proof}

\section{Concluding remarks}\label{discussion}

 Since the introduction of projection depth \cite{ZuoSerfling}, it has been applied in multivariate analysis (see, e.g., \cite{dutta} and \cite{Zuoprojection}). It measures the worst case of the outlyingness of a point, comparing the projection of the point in every direction with respect to the univariate median of the projection in that direction.
In the fuzzy case, as the support function of a fuzzy set considers the projection for every direction $u$ and every $\alpha$-level, we define the function $D_{FP}$ replacing the inner product by the support function for every $(u,\alpha)\in\mathbb{S}^{p-1}\times [0,1]$.

The function $D_{FP}$ is the natural generalization of the multivariate projection depth to the fuzzy setting (Proposition \ref{promult}). Projection depth for fuzzy sets, as the Tukey depth defined in \cite{primerarticulo}, is a semilinear depth function and also a geometric depth function for the $\rho_{r}$-distances with $r\in(1,\infty)$ (Corollary \ref{coropro}). It is also interesting that, being defined via medians, it imposes no integrability requirements on the fuzzy random variables. In summary,   projection  depth is a nice alternative to  Tukey depth in the fuzzy setting.

For any $r\in [1,\infty),$ the $L^{r}$-type fuzzy depths  satisfy  the semilinear and the  geometric depth notions under the assumption that the matrices considered in P1  are orthogonal  (Proposition \ref{resultadoDr1}). 
Property P2 is satisfied by $D_{1} = RD_{1}$ and $RD_{2}$ when $F$-symmetry is considered (see Proposition \ref{resultadoDr2} and \ref{Simetriar2}) and by $D_{r}$ for $r\in [1,\infty)$ when \textit{functional symmetry} is considered (see Theorem \ref{simetriaLr}). It is also satisfied by $D_{1}^{\theta} = RD_{1}^{\theta}$ and $RD_{2}^{\theta}$ for $\theta\in [0,\infty)$ when $(\midd,\spr)$-symmetry is considered (see Proposition \ref{P2Dtheta} and \ref{P2Dtheta2}). Although $L^{r}$-type depths are neither semilinear nor geometric depth functions, we can observe in Section \ref{simulations} that their behavior can be similar to that of projection depth, which is in fact a semilinear and a geometric depth function.

For future work, it would be desirable to study the continuity or semicontinuity properties of these depth functions, as it is done in the multivariate case (see \cite{ZuoSerfling}) and in the functional case (see \cite{NietoBattey}). It is also open to find a geometric depth function for the $\rho_{1}$ metric  or $d_{r}$ metrics, or to show the impossibility of such functions. From the point of view of applied mathematics, it could be stimulating to develop algorithms to compute some fuzzy depth proposals, in order to generalize to fuzzy sets some nonparametric methods of multivariate and functional data analysis.

\vskip 1 true cm
{\bf Acknowledgments}
The authors are supported by grant PID2022-139237NB-I00 funded by MCIN/AEI/10.13039/501100011033 and “ERDF A way of making Europe”. Additionally, L. Gonz\'alez was supported by the Spanish Ministerio de Ciencia, Innovaci\'on y Universidades grant MTM2017-86061-C2-2-P. 
P. Ter\'an is also supported by the Ministerio de Ciencia, Innovación y Universidades grant PID2019-104486GB-I00.

\end{document}